\documentclass[12pt, notitlepage]{amsart}
\usepackage{latexsym, amsfonts, amsmath, amssymb, amsthm, cite}

\usepackage{stmaryrd}

\newtheorem{theorem}{Theorem}[section]
\newtheorem{lemma}[theorem]{Lemma}
\newtheorem{proposition}[theorem]{Proposition}

\newtheorem{corollary}[theorem]{Corollary}

\newtheorem{definition}[theorem]{Definition}

\usepackage{graphicx}
\usepackage{mathrsfs}
\numberwithin{equation}{section}

\begin{document}

\title{Integral factorial ratios} 
 \author{K. Soundararajan} 
 \address{Department of Mathematics \\ Stanford University \\
450 Serra Mall, Bldg. 380\\ Stanford, CA 94305-2125}
\email{ksound@math.stanford.edu}
\thanks{} 

 \maketitle

\section{Introduction}  

\noindent This paper is motivated by the following problem: classify all tuples of natural numbers $a_1$, $\ldots$, $a_K$, and 
$b_1$, $\ldots$, $b_L$ with $\sum_{i} a_i = \sum_{j} b_j$ such that 
$$ 
\frac{(a_1 n)! (a_2 n)! \cdots (a_K n)!}{(b_1 n)! (b_2 n)! \cdots (b_L n)!} \in {\Bbb N} 
$$ 
for all natural numbers $n$.   The condition $\sum a_i = \sum b_j$ ensures that these ratios grow only exponentially 
with $n$, so that the power series formed with these coefficients is a hypergeometric series.  Naturally one may assume that $a_i \neq b_j$ for all 
$i$ and $j$.  As will become apparent shortly, one may also assume that $L>K$ (else there are no solutions), and assume that the gcd of 
$(a_1, \ldots, a_K, b_1, \ldots, b_L) $ is $1$ (and we call tuples with this gcd condition {\sl primitive}).

We have in mind the situation when $D = L-K$ is a fixed positive integer, and the problem is to determine all possible integral 
factorial ratios as above with $D$ more factorials in the denominator than the numerator.   Multinomial coefficients are a natural source of 
examples, but there are also exotic examples such as Chebyshev's observation (used in obtaining estimates for the number of primes below $x$) that 
$$ 
\frac{(30n)! n!}{(15n)! (10n)! (6n)!} 
$$ 
is an integer for all $n$.   In full generality, the problem of classifying these integral factorial ratios is wide open, and only the case $D=1$ 
has been resolved completely.  

\begin{theorem} \label{thm1.1} In the case $D=1$, there are three infinite families of primitive integral factorial ratios: 
$$
\frac{((a+b)n)!}{(an)! (bn)!}, \ \  \frac{(2an)! (2bn)!}{(an)! (bn)! ((a+b)n)!}, \ \ \frac{(2an)! (bn)!}{(an)! (2bn)! ((a-b)n)!}, 
$$ 
where $a$ and $b$ are coprime natural numbers, and with $a>b$ in the third family.  Besides these there are fifty two sporadic examples 
of primitive integral factorial ratios: $29$ examples with $K=2$ and $L=3$; $21$ examples with $K=3$ and $L=4$; and $2$ examples with $K=4$ and $L=5$.  
\end{theorem}  

Theorem \ref{thm1.1} was established by Bober \cite{Bober}, building on a key observation of Rodriguez--Villegas \cite{RV}.  Rodriguez--Villegas noted 
that when $D=1$ the condition that the factorial ratios are integers is equivalent to the associated hypergeometric function being algebraic.  In turn, the 
condition that the hypergeometric function is algebraic amounts to the finiteness of its associated monodromy group.  The work of Beukers and Heckman \cite{BH} 
describes all hypergeometric series ${}_nF_{n-1}$ with finite monodromy group, and thus enables the classification given in Theorem \ref{thm1.1}.   At the 2018 Bristol 
conference {\sl Perspectives on the Riemann Hypothesis}, Rodriguez--Villegas asked for a more direct proof of Theorem \ref{thm1.1}, and one of our main goals is to describe an elementary, self-contained proof.  
In addition, we can impose some restrictions on the possible integral factorial ratios in the situation $D>1$.  

\begin{theorem} \label{thm1.2}  Suppose $D\ge 1$ and put $L=K+D$.  

1.  There are no integral factorial ratios with $K$ terms in the numerator and $L$ in the denominator unless
$$ 
K+L \le (1+o(1)) \frac{18e^{2\gamma}}{\pi^2} D^2 (\log \log (D+2))^2. 
$$

2.  The points $(a_1,\ldots, a_K,b_1,\ldots, b_L) \in {\Bbb R}^{K+L}$ corresponding to an integral factorial ratio lie on finitely many 
vector subspaces of ${\Bbb R}^{K+L}$ of dimension at most $3D^2-1$.  

3.  In the special case $D=2$, there are no integral factorial ratios if $K+L \ge 82$, and if $K+L \ge 76$ there are at most finitely many primitive 
solutions. 
\end{theorem}

Part 1 of Theorem \ref{thm1.2} refines earlier results of Bell and Bober \cite{BB}  (see also the unpublished work of Bombieri and Bourgain \cite{BoBo}) who 
obtained the bound $K+L \ll D^2 (\log D)^2$; this established a conjecture of Borisov \cite{Bor}.   In his senior thesis advised by me, Schmerling \cite{Sch} made 
improvements to Bell and Bober's argument and obtained $K+L \ll D^2(\log \log D)^2$.  Our proof of part 1 of Theorem \ref{thm1.2} follows a different, more combinatorial, approach and forms 
part of our proof of Theorem \ref{thm1.1}.   In the case $D=1$, Bell and Bober established the bound $K+L < 112371$, and Schmerling obtained the much better bound $K+L <43$; 
we will first see here that $K+L <11$ and then give the classification in the remaining cases.  In the case $D=2$, the bound of Bell and Bober is $K+L < 502827$, while Schmerling 
gives $K+L < 202$.  Here it is expected that $K+L \le 18$, so that our bound too is far from the truth.  

Regarding part 2 of Theorem \ref{thm1.2}, Bober \cite{Bober2} established that the points $(a_1,\ldots, a_K, b_1, \ldots, b_L)$ are exactly the points with 
natural number coordinates in a finite union of vector subspaces of ${\Bbb R}^{K+L}$.  We have given an upper bound for the dimension of these spaces.  
It is likely that our upper bound is not close to the truth, which perhaps is linear in $D$.

 We now describe our framework for establishing these results.   In general, all known approaches to establishing that factorial 
 ratios are integral proceed by showing that for every prime $p$, the power of $p$ dividing the numerator is at least the power dividing the 
 denominator.  This readily leads to the equivalent formulation (going back to Landau) that for all real $x$ one must have 
 $$ 
 f(x;{\bold a}, {\bold b}) := \sum_{i=1}^{K} \lfloor a_i x\rfloor - \sum_{j=1}^{L} \lfloor b_j x \rfloor \ge 0. 
 $$ 
 This equivalent condition makes clear why we may assume that the gcd of $(a_1,\ldots, a_K, b_1, \ldots, b_L)$ equals $1$. 
 Further, since (apart from some rational numbers) $f(-x;{\bold a}, {\bold b}) = (L-K) - f(x;{\bold a}, {\bold b})$ we see that $D=L-K$ must be 
 positive and that $f(x;{\bold a}, {\bold b})$ must take values in $\{ 0, 1, \ldots, (L-K)=D\}$.  
 
 The work of Bell and Bober, and Schmerling now proceeds by obtaining lower bounds (in terms of $K+L$) 
 for $\int_0^1 (f(x; {\bold a}, {\bold b}) -D/2)^2 dx$.   We will also study this quantity, but will develop a different 
 inductive approach to bounding it from below.  In order to set up our induction framework, we generalize the problem a little.

Throughout we let $\{x\} = x -\lfloor x\rfloor$ denote the fractional part of $x$, and let  
$$ 
\psi(x) = 1/2 - \{ x\}    
$$ 
denote the ``saw-tooth function".    Let ${\frak a} = [a_1, \ldots, a_n]$ denote a list of $n$ non-zero integers.  Associated to such a list is the $1$--periodic function 
\begin{equation} 
\label{1.1} 
{\frak a}(x) = \sum_{j=1}^{n} \psi(a_j x). 
\end{equation} 
    Our interest is in the norm of such a list, defined by 
\begin{equation} 
\label{1.2} 
N([a_1,\ldots,a_n]) = N({\frak a}) = \int_0^1 {\frak a}(x)^2 dx. 
\end{equation}  
We call a list {\sl degenerate} if it contains both $a$ and $-a$ for some integer $a$.   In such a case, the function ${\frak a}(x)$ and the norm $N(\frak a)$ 
are unaltered if both $a$ and $-a$ are removed from the list.   Thus, henceforth we shall restrict attention to {\sl non-degenerate} lists, and for such a list $\frak a$ 
we let $\ell(\frak a)$ denote the length of this list, and $s(\frak a)$ denote the sum of the elements in the list, $a_1 +\ldots +a_n$.   Given a non-zero integer 
$k$, we denote by $k\frak a$ the list obtained by multiplying all elements of $\frak a$ by $k$.  Note that $(k\frak a)(x) = {\frak a}(kx)$, and that $N(k\frak a) = N(\frak a)$.  
We say that a (non-degenerate) list is {\sl primitive} if the elements of the list have gcd $1$, and in studying the norm we may clearly restrict attention to primitive lists. 
We will also treat all permutations of the entries of a list as being the same; clearly the associated functions and norms are unaltered.

If $(a_1,\ldots,a_K, b_1,\ldots, b_L)$ is a tuple giving rise to an integral factorial ratio, then we associate to this tuple the 
list of length $K+L$ given by $\frak a = [a_1,\ldots,a_K, -b_1, -b_2, \ldots, -b_L]$.  Note that, in our earlier notation, 
$$ 
\frak a(x) = f(x;\bold a, \bold b) - D/2, 
$$ 
so that the factorial ratio being integral is equivalent to $\frak a(x)$ taking the values from $-D/2, -D/2+1, \ldots, D/2$.  In particular one must 
have $N(\frak a) \le D^2/4$, and this motivates our study of lower bounds for the norms of lists in general.   Note that the lists 
arising from factorial ratios also satisfy $s(\frak a) = 0$, but for the inductive argument we have in mind it is convenient to allow more 
general lists.

Given $n\ge 1$, a central object in our study is 
\begin{equation} 
\label{1.3} 
G(n) = \inf \{ N(\frak a): \ \ \ell(\frak a) = n \}, 
\end{equation} 
where the infimum is taken over all non-degenerate lists of length $n$.  To prove Theorem \ref{thm1.1} we shall find $G(n)$ 
for small $n$, as well as classify those lists with suitably small norm.   More generally, given $d<n$ we shall also consider 
\begin{equation} 
\label{1.4} 
G(n;d) = \sup_{\substack{V_1, \ldots, V_r \\ \text{dim}(V_j) \le d}} \inf \{ N(\frak a): \ \ell(\frak a)=n ,  \ \frak a \notin V_1 \cup \ldots \cup V_r \}. 
\end{equation} 
Here the supremum is taken over all finite collections of vector subspaces $V_1$, $\ldots$, $V_r$ of ${\Bbb R}^n$ with dimension at most 
$d$, and the infimum is taken over all non-degenerate $\frak a$ of length $n$ not lying in one of these subspaces.  Note that $G(n;0)$ is 
simply $G(n)$, and that $G(n;1)$ is the infimum of $N(\frak a)$ after removing any finite number of primitive lists of length $n$. If $d\ge n$ then 
we set $G(n;d) =\infty$.   In addition, we define the analogues of $G(n)$ and $G(n;d)$ when restricted to the hyperplane of  lists that sum to $0$.  Thus, for $n>d$, we 
put 
\begin{equation} 
\label{1.5} 
{\widetilde G}(n;d) = \sup_{\substack{V_1, \ldots, V_r \\ \text{dim} V_j \le d } } \inf \{ N(\frak a), \ s(\frak a)= 0, \ \frak a\notin V_1 \cup \ldots \cup V_r \},
\end{equation} 
with ${\widetilde G}(n):= {\widetilde G}(n;0)$.  

\begin{theorem} \label{thm1.3} For $0\le d\le n-1$ we have 
$$ 
G(n;d) = \min_{\substack{ \ell_1 + \ldots + \ell_{d+1} = n\\ \ell_j \ge 1}} \big( G(\ell_1) + \ldots + G(\ell_{d+1}) \big).  
$$ 
Further for $0 \le d\le n-2$ we have 
$$ 
{\widetilde G}(n;d) \ge \min\Big( G(n;d+1), \min_{\substack{\ell_1 +\ldots + \ell_{d+1}=n \\ \ell_j \ge 1} } {\widetilde G}(\ell_1) +\ldots + {\widetilde G}(\ell_{d+1}) \Big). 
$$ 
In particular, for all $2d+2 \ge n \ge d+1$ we have $G(n;d) = (d+1)/12$, and if $n\ge 2d+3$ then $G(n;d) \ge d/12 + 1/9$.  Further if $d+2 \le n \le 2d+4$ then 
${\widetilde G}(n;d) \ge (d+2)/12$, while if $n\ge 2d+5$ then ${\widetilde G}(n;d) \ge d/12 +7/36$.  
\end{theorem}

In our work we shall determine $G(n)$ for all $2\le n\le 8$, and obtain precise explicit lower bounds for $G(n)$.  In particular, we can 
determine the asymptotic nature of $G(n)$ for all large $n$.  

\begin{theorem} \label{thm1.4}  For large $n$ one has 
$$ 
G(n) \sim \frac{\pi^2 }{72e^{2\gamma}} \frac{n}{(\log \log n)^2}. 
$$ 
\end{theorem}  

The paper is organized as follows.  In Section 2 we introduce the new notion of $k$--separatedness of lists, which is used 
in Section 3 to set up an inductive framework for obtaining lower bounds for $G(n)$.  In Section 4 we consider lists of length $2$, $3$ 
and $4$, and classify such lists with small norm.  Section 5 lays the foundations for the proof of Theorem \ref{thm1.1}, and gives 
a qualitative version (see Theorem \ref{thm5.1}) showing that there are only finitely many sporadic examples.  The complete classification 
of sporadic examples  is carried out in Section 6 (for length $5$), Section 8 (for length $7$) and Section 10 (for length $9$).  These 
involve a precise understanding of lists of length $5$, $6$, $7$ and $8$ with small norm, which is carried out in Sections 7 and 9.  In determining 
lists of small norm, and classifying the sporadic examples of Theorem \ref{thm1.1} we made use of computer calculations.  The programs were written in Python, 
and all the computations were carried out using fractions so that the answers are exact and not approximate.  The programs are not very involved, and most of them 
executed in a manner of minutes, with the slowest step taking about an hour.  Sections 11, 12, and 13 treat Theorems \ref{thm1.3}, \ref{thm1.4}, and \ref{thm1.2}.

We end the introduction by mentioning some related work.   In the case $D=1$, 
the paper of Vasyunin \cite{V}, motivated by the Nyman--Beurling formulation of the 
Riemann Hypothesis, identified the infinite families as well as the fifty two sporadic examples, and 
conjectured that these are all the examples.  

 The integrality of factorial ratios can be reformulated 
as a question on interlacing sets of fractions in $[0,1)$; see Bober \cite{Bober}.  Indeed this interlacing 
condition is what arises naturally in the work of Beukers and Heckman \cite{BH}.  A variant of this interlacing 
condition is studied in \cite{FMS} in connection with the problem of determining which hyperbolic hypergeometric 
monodromy groups are thin.  In passing, the authors there mention that one obtains examples of factorial ratios 
with $D=3$.  However, the examples so obtained are imprimitive in the sense that they arise from one of 
the sporadic examples with $D=1$ multiplied by two binomial coefficients.  

The masters thesis of Wider \cite{W} gives some examples of factorial ratios with $D\ge 2$, and considers 
the problem of determining whether such factorial ratios are primitive or not.  For example, Wider 
gives the family $[3a, 3b, -a, -b, -(a+b), - (a+b)]$, which is primitive, but several of the other examples 
found by him turn out to be imprimitive.   Our work on lists with small norm suggests a possible approach to finding 
examples of such families.  For example, we found that $[-a, 2a, -4a, -b, 2b, -4b, 6(a+b), -3(a+b)]$ 
gives a two parameter family of factorial ratios with $D=2$.  We hope to discuss such examples elsewhere.  

The complementary problem of finding sets of $n$ positive integers $a_1$, $\ldots$, $a_n$ with maximal value of 
$\sum_{i,j=1}^{n} (a_i, a_j)^2/(a_ia_j)$  is considered in Lewko and Radziwi{\l \l} \cite{LR} who establish an analogue of Theorem \ref{thm1.4} in this 
context.  In a sense, their argument is closely related to the approach in \cite{Sch} and could be adapted to give an alternative proof of Theorem \ref{thm1.4}; 
however the bounds obtained in this way would not be sharp enough to yield Theorem \ref{thm1.1} and the combinatorial approach developed here yields 
sharper lower bounds for $G(n)$ for small $n$.

\smallskip 

\noindent {\bf Acknowledgements.}   I am grateful to Peter Sarnak for encouragement and helpful remarks, and to him and Amit Ghosh 
for drawing my attention to the work of Wider.  I am partially supported by a grant from the National Science Foundation, and a Simons Investigator grant 
from the Simons Foundation.

\section{The key notion: $k$--separatedness of lists} 


\noindent In this section we introduce the key property of being $k$--separated which will set up 
an inductive procedure to evaluate norms of lists.  First, let us recall  that 
for any two non-zero integers $a$ and $b$  one has 
$$ 
\int_0^1 \psi(ax) \psi(bx) = \frac{1}{12} \frac{(a,b)^2}{ab}.  
$$ 
This is easily established; for example, by using the Fourier expansion 
$$ 
\psi(x) = 1/2 -\{ x\} =  \sum_{n\neq 0} \frac{e^{2\pi inx}}{2\pi in}, 
$$ 
and applying the Parseval formula.   Therefore 
\begin{equation} 
\label{2.1} 
N([a_1,\ldots,a_n]) = \frac{1}{12} \sum_{i,j=1}^{n} \frac{(a_i, a_j)^2}{a_i a_j}. 
\end{equation}

Given two lists ${\frak a}_1$ and ${\frak a}_2$, we denote by ${\frak a}_1 + {\frak a}_2$ the list obtained by concatenating these two lists, and removing any degeneracies.  
Note that even if ${\frak a}_1$ and ${\frak a}_2$ are non-degenerate, concatenating them might result in degeneracies, which are removed in defining ${\frak a}_1 +{\frak a}_2$.  
The next definition gives the key tool in our analysis.  

\begin{definition} \label{def2.1}   Let $k\ge 2$ be a natural number.  A primitive (non-degenerate) list $\frak a$ (of length $n$) is said to be $k$--separated, or $k$--separated of type $(\ell, m)$, if there are two primitive lists $\frak b$ and $\frak c$ with $1\le \ell =\ell(\frak b) \le m=\ell(\frak c) < n$ with $\ell+m =n$ such that: 

1.  There are two non-zero coprime integers $B$ and $C$ such that 
$$ 
\frak a = B \frak b +C\frak c. 
$$ 

2.  Exactly one of $B$ or $C$ is a multiple of $k$ (and the other is therefore coprime to $k$).  

3.  If $k|B$ then we set ${\widetilde {\frak b}} = (B/k) \frak b$ and $\widetilde{\frak c} = C\frak c$, and require that for all $kb\in B\frak b$ and $c\in \frak c$ one has $(kb,c) = (b,c)$.  Similarly if $k|C$ then we set $\widetilde{\frak b} = B\frak b$ and $\widetilde{\frak c} = (C/k) \frak c$ and require that for all $kc \in C\frak c$ and all $b\in \frak b$ one 
has $(b,kc) = (b,c)$. 
\end{definition}

 In all our work below, when a list $\frak a$ is $k$--separated, the symbols $\frak b$, $\frak c$, $B$, $C$, $\widetilde{\frak b}$ and $\widetilde{\frak c}$ will 
have the meanings assigned in the above definition.   Note that part of the definition  includes that there are no degeneracies when concatenating $B\frak b$ and $C\frak c$.  

The key point of this definition is the gcd condition imposed in part 3.  To help with this condition, we give a few illustrations.   
If the list ${\frak a}$ has at least one multiple of a prime $p$, and one non-multiple of a prime $p$, 
then ${\frak a}$ is $p$--separated, and we can split it as the list of multiples of $p$  and the list of non-multiples of $p$.   
A list $\frak a$ can be $k$--separated in several different ways:  for example $[1, -p, p^2]$ can be split as $p\times [-1,p] + [1]$ or as $p \times [p] + [1,-p]$.  
Lastly, the reader may check that $[30, -15, -10, - 6, 1]$ is $2$, $3$, and $5$--separated, but is not $6$--separated.

The following proposition calculates the norm of $\frak a$ in terms of smaller lists $\frak b$ and $\frak c$, thereby setting the stage 
for inductive arguments. 

\begin{proposition} \label{prop2.2}  Suppose that ${\frak a}$ is $k$--separated, and splits as $\frak a = B \frak b + C\frak c$ as in Definition \ref{def2.1}.  
Then 
$$ 
N(\frak a )= \Big( 1- \frac 1k \Big) \Big( N(\frak b) + N(\frak c) \Big) + \frac 1k N(\widetilde{ \frak b} + \widetilde{\frak c}). 
$$
Further, note that $\ell(\widetilde{\frak b}+\widetilde{ \frak c}) \ge |\ell(\frak b) -\ell (\frak c)|$, and that $\ell(\widetilde{\frak b} +\widetilde{ \frak c})$ has the same parity as $\ell(\frak a)$.  
\end{proposition}
\begin{proof}  Suppose $k|B$ so that $\widetilde{\frak b} = (B/k) \frak b$ and $\widetilde{\frak c} = C \frak c$.  Note that 
$$ 
N(\frak a) = \int_0^1 (B{\frak b} + C {\frak c})(x)^2 dx = N(\frak b) + N(\frak c) + 2 \int_0^1 (B{\frak b})(x) (C{\frak c})(x) dx.
$$ 
Since $(kb,c)= (b,c)$ for any elements $kb$ of $B{\frak b}$ and $c$ of $C\frak c$, it follows that 
$$ 
\int_0^1 (B\frak b)(x) (C\frak c)(x) dx =\frac 1{12} \sum_{\substack{kb \in B\frak b\\ c\in \frak c}} \frac{(kb,c)^2}{kbc} = 
\frac{1}{k} \frac{1}{12} \sum_{\substack{b\in \widetilde{\frak b} \\ c\in \widetilde{\frak c}}} \frac{(b,c)^2}{bc} 
= \frac{1}{k} \int_0^1 {\widetilde {\frak b}}(x) \widetilde{\frak c}(x) dx. 
$$
Since $N(\frak b) = N(\widetilde{\frak b})$ and $N(\frak c) = N(\widetilde{\frak c})$ we conclude that 
\begin{align*}
N(\frak a) &=  \Big( 1-\frac 1k\Big) (N(\frak b) + N(\frak c)) + \frac 1k \Big( N(\widetilde{\frak b}) + N(\widetilde{\frak c}) + 2\int_0^1 \widetilde{\frak b}(x) 
\widetilde{\frak c}(x) dx \Big) \\
&= 
\Big(1-\frac 1k \Big) (N(\frak b) + N(\frak c)) + \frac 1k N(\widetilde{\frak b} +\widetilde{\frak c}). 
\end{align*} 
The case $k|C$ follows in exactly the same way. 

When the lists $\widetilde{\frak b}$ and $\widetilde{\frak c}$ are added, the number of degeneracies that must be removed is always even 
and is at most $2\min (\ell(\widetilde{\frak b}), \ell(\widetilde{\frak c}))$.  This proves the last assertion of the proposition. 
\end{proof} 

The other important feature of our definition is that only finitely many primitive lists of length $n$ are at most $k$-separated.  
 Naturally we say that ${\frak a}$ is {\sl at most $k$--separated} if it is not $\ell$--separated for any $\ell >k$, and 
we say that ${\frak a}$ is {\sl at least $k$--separated} if it  is $\ell$--separated for some $\ell \ge k$.

 \begin{lemma} \label{lem2.3}  Let $p$ be a prime, and let $\frak a$ be a primitive list of length $n$.  Suppose there are 
 two elements $a_1$ and $a_2$ in $\frak a$ with $p^{e_1} \Vert a_1$ and $p^{e_2} \Vert a_2$.  If $e_2 \ge e_1+r$, and there 
 is no element of $\frak a$ exactly divisible by $p^{e}$ with $e_1 < e < e_2$, then $\frak a$ is $p^r$--separated. 
 \end{lemma} 
\begin{proof}  Divide $\frak a$ into the two non-empty lists consisting of the multiples of $p^{e_2}$ and those elements 
that are not multiples of $p^{e_2}$.  In the notation of our definition, $B\frak b$ will be the smaller of these two lists, and $C\frak c$ the 
longer one.  The important gcd condition of the definition (for being $p^r$--separated) is satisfied because all the elements of one list will be multiples of $p^{e_2}$ 
while all the elements of the other list are divisible at most by $p^{e_1}$.
\end{proof} 
 
From the lemma it is easy to deduce that a primitive list $\frak a$ of length $n$ that is at most $k$--separated must consist of divisors of a specified 
number.     

\begin{proposition} \label{prop2.4}   If $\frak a$ is a primitive list of length $n$ that is at most $k$--separated, then the elements of $\frak a$ are divisors of 
$$ 
\prod_{\substack{p^r \le k< p^{r+1}} } p^{r(n-1)}, 
$$ 
where the product is over all primes $p \le k$ and $p^r$ is the largest power of $p$ at most $k$. 
\end{proposition} 
\begin{proof}  Suppose $p$ is prime and $p^r \le k < p^{r+1}$.  Write in ascending order the sequence of the powers of $p$ dividing $a_j$.  Since $\frak a$ is 
primitive, there is some $a_j$ that is not a multiple of $p$ and so this sequence starts with $0$.  If the sequence ends in a number larger than $r(n-1)$ then by the pigeonhole principle there must be two consecutive exponents that differ by at least $(r+1)$.   But then by Lemma \ref{lem2.3} we would know that $\frak a$ is $p^{r+1}$--separated.  
\end{proof}

In our future work we shall use Proposition \ref{prop2.4}, together with variants when the list is known to be of a special form, to consider (using a computer calculation) 
all possible lists of length $n$ and separation at most $k$.   

\begin{definition} \label{def2.5}   If $n$ is even then a list $\frak a$ is said to be of Type A if it is of the form $[a_1, -2a_1, a_3, -2a_3, \ldots, a_{n-1}, -2a_{n-1}]$.  If $n$ is odd, then $\frak a$ is of Type A if it is of the form $[a_1, -2a_1, 
a_3, -2a_3, \ldots, a_{n-2}, -2a_{n-2},a_n]$.   A list is said to be of Type B if it is not of Type A.  
\end{definition}  

Most of the lists that give rise to integral factorial ratios (with $D=1$) are of Type A, and Type A lists also account for many of the lists with small norm.   The special shape 
of Type A lists, however, allows us to search over lists with greater separation than we could for the most general lists.   Rather than stating general results (which would be 
proved exactly as in Proposition \ref{prop2.4}) we content ourselves with giving some typical examples.   

Consider primitive lists $\frak a$ of length $7$ that are at most $7$ separated.  Proposition \ref{prop2.4}  gives that the elements of $\frak a$ must be divisors of 
$2^{12} \times 3^6 \times 5^6 \times 7^6$.  If we restrict to lists $\frak a$ that are of Type A, then the elements of $\frak a$ are constrained to be divisors of $2^9 \times 3^3 \times 5^3 \times 7^3$.   Suppose now that $\frak a$ is further restricted to be of Type A and to satisfy  $s(\frak a) = 0$.   If we write $\frak a =[a,-2a, b, -2b, c, -2c, d=a+b+c]$ then at least two of $a$, $b$, $c$, $d$ 
must be coprime to $p$ for each prime $p\le 7$.  Using this fact, and arguing as in Proposition \ref{prop2.4}, we may see that the elements of $\frak a$ are now 
forced to be divisors of $2^6 \times 3^2 \times 5^2 \times 7^2$.

\section{General lower bounds on $G(n)$} 


\noindent  In this section we shall establish bounds on norms of lists of length $n$, by induction on $n$ as 
well as induction on the largest prime factor of the elements of the list.   Here it is convenient to denote the $j$-th 
prime by $p_j$, and to define 
$$
G_r (n) = \inf \{ N({\frak a}): \ \frak a = [a_1, \ldots, a_n] , \ \ p|(a_1 \cdots a_n) \implies p \le p_r \}, 
$$ 
so that $G_r(n)$ gives the infimum of norms of (non-degenerate) lists all of whose elements are composed only of the first $r$ primes.   
The two lists of length $n$  consisting of all $1$'s or all $-1$'s have norm $n^2/12$, and so it is natural to define $G_0(n)$ as $n^2/12$.  
Clearly $G_0(n) \ge G_1(n) \ge G_2(n) \ge G_3(n) \ge \ldots \ge G(n)$, and indeed $G(n)= \lim_{r\to \infty} G_r(n)$.   
Note also that when $n=1$, $N([a])=1/12$ for all non-zero integers $a$, and so $G_r(1)=1/12$ for all $r$.   

\begin{proposition} \label{prop3.1}  For all $n\ge 2$ and $r\ge 1$ we have 
$$ 
G_r(n) \le \min_{1\le i < n} (G_r(i) + G_{r-1}(n-i)). 
$$ 
Moreover 
$$ 
G(n) =G(n;0) \le G(n;1)  = \min_{1 \le i <n} (G(i) +G(n-i)).
$$ 
Finally, there is a non-degenerate list of length $n$ attaining the norm $G(n)$, so that the infimum in the definition of $G(n)$ is an 
attained minimum. 
\end{proposition} 
 \begin{proof}  Suppose $\frak b$ and $\frak c$ are two lists with lengths $i$ and $n-i$ respectively.  Let $p$ be a prime 
 larger than the largest prime factor of elements in the list $\frak c$.   Then the lists $p\frak b + \frak c$ and $p\frak b - \frak c$ are 
 both of length $n$ and $p$--separated (in the obvious way), and one has 
 \begin{equation} 
 \label{3.1} 
 N(p\frak b \pm \frak c) = N(\frak b) + N(\frak c)  \pm \frac{2}{p} \int_0^1 \frak b(x) \frak c(x) dx. 
 \end{equation} 
 It follows that at least one of the two lists $p\frak b \pm \frak c$ has norm at most $N(\frak b) + N(\frak c)$.  If we let $p$ tend to infinity here, then 
 the lists $p\frak b\pm \frak c$ have norm tending to $N(\frak b) + N(\frak c)$.  
 
Using this observation to a list $\frak b$ with all elements composed of primes at most $p_r$, and a list $\frak c$ with all elements 
composed of primes at most $p_{r-1}$, and with $p=p_r$ we obtain that $G_r(n) \le G_r(i) + G_{r-1}(n-i)$, and the first claim of the 
proposition follows. 

For any $\epsilon >0$ we may find $\frak b$ of length $i$ with  $N(\frak b) \le G(i) + \epsilon$ and $\frak c$ of length $n-i$ with $N(\frak c) \le G(n-i) + \epsilon$, and 
so by choosing $p$ sufficiently large, we find that there are infinitely many primitive lists $\frak a= p \frak b +\frak c$ with norm below $G(i) + G(n-i) + 3\epsilon$.  
Since  $\epsilon >0$ is arbitrary, we conclude that $G(n;1) \le \min_{1\le i < n} (G(i) + G(n-i))$.  

To obtain the reverse inequality, suppose that there is an infinite sequence of primitive lists $\frak a_j$, all of length $n$, with 
$N(\frak a_j)$ converging to $G(n;1)$.   By Proposition \ref{prop2.4}, given any $k$, if $j$ is large enough then ${\frak a}_j$ is at least $k$--separated.  
Appealing now to Proposition \ref{prop2.2} we find that $N(\frak a_j) \ge (1-1/k) \min_{1\le i <n} (G(i) + G(n-i))$.   Letting $k \to \infty$ we conclude that $G(n;1) \ge \min_{1\le i<n} (G(i)+G(n-i))$, as claimed.  

Lastly, it remains to show that there is a list of length $n$ with norm $G(n)$, which we establish by induction.  The length $1$ case is trivial, and suppose the claim holds 
for all lengths below $n$.  If $G(n) < G(n;1)$ then (by the definition of $G(n;1)$) there are only finitely many primitive lists of length $n$ with norm below $(G(n)+G(n;1))/2$, 
and therefore a minimum exists.  If $G(n)=G(n;1)$, then pick (using the induction hypothesis) two primitive lists $\frak b$ and $\frak c$ with lengths adding up to $n$ and 
with $N(\frak b) + N(\frak c) = G(n;1)$.   For large enough $p$, using \eqref{3.1} and the assumption that there are no lists with norm below $G(n;1)$, we find that  
$\int_0^1 \frak b(x) \frak c(x)dx =0$ and that the lists $p\frak b \pm \frak c$ all have norm $N(\frak b) + N(\frak c)$.   This completes the proof. 
\end{proof}

\begin{proposition} \label{prop3.2}  For all $r\ge 1$ and $n\ge 2$ we have 
\begin{equation} 
\label{3.2} 
G_r(n) \ge \min_{1 \le i<n} \min_{\substack{|n-2i| \le j < n \\ j-n \text{ even}} } \Big( G_{r-1}(n),  G_r(i) + G_{r-1}(n-i), \Big(1-\frac 1{p_r}\Big) 
\Big( G_r(i)  + G_{r-1}(n-i) \Big) + \frac{G_{r} (j)}{p_r} \Big). 
\end{equation} 
Further, for any $r\ge 0$ we have   
\begin{equation} 
\label{3.3} 
G(n) \ge \min_{1 \le i<n} \min_{\substack{|n-2i| \le j < n \\ j-n \text{ even}} } \Big( G_r(n), G(i) + G(n-i), \Big( 1- \frac{1}{p_{r+1}}\Big) \Big( 
G(i) + G(n-i)\Big) +\frac{G(j)}{p_{r+1}}\Big). 
\end{equation} 
\end{proposition} 
\begin{proof}  We begin with \eqref{3.2}.  Suppose ${\frak a}$ is a primitive list of length $n$ all of whose prime factors are at most $p_r$.  If there 
is no element of the list divisible by $p_r$, then $N(\frak a) \ge G_{r-1}(n)$ and the desired inequality holds.   So we may assume that 
$\frak a$ has at least one element being a multiple of $p_r$ and one that is not.  Thus $\frak a$ is $p_r$--separated, and we write $\frak a = p_r  \frak b + \frak c$ 
with $\frak c$ denoting the elements of $\frak a$ not divisible by $p_r$.   Thus by Proposition \ref{prop2.2}  
$$ 
N(\frak a) = \Big(1-\frac 1{p_r} \Big) \Big( N(\frak b) + N(\frak c) \Big) + \frac 1{p_r} N(\frak b +\frak c). 
$$ 
Since $\frak b$ and $\frak b+\frak c$ are lists with all elements divisible by primes at most $p_r$, and $\frak c$ is a list with 
elements divisible by primes at most $p_{r-1}$, it follows that 
$$ 
G_{r}(n) \ge \min_{1 \le i< n} \Big( \Big( 1- \frac 1{p_r}\Big) \Big( G_r(i) + G_{r-1}(n-i)\Big) + \frac 1{p_r} \min_{\substack{|n-2i| \le j \le n \\ j\equiv n \bmod 2}} G_r(j)\Big). 
$$ 
Upon considering whether the minimum over $j$ above occurs for $j=n$, or for a smaller value of $j$, we obtain \eqref{3.2}.  

The proof of \eqref{3.3} is similar.  Let  ${\frak a}$ be a primitive list of length $n$.  If all elements of ${\frak a}$ are divisible only by primes at most $p_r$ then $N(\frak a) \ge G_r(n)$.  
Otherwise, for some prime $p\ge p_{r+1}$, we may split $\frak a$ as $p \frak b + \frak c$ where $\frak b$ and $\frak c$ are non-empty lists and the elements of $\frak c$ are 
all not divisible by $p$.  By Proposition \ref{prop2.2} we obtain 
$$ 
N(\frak a) = \Big( 1-\frac 1p\Big) \Big(N(\frak b) + N(\frak c) \Big) + \frac 1p N(\frak b+\frak c) 
\ge \min_{1\le i <n} \min_{\substack{ |n-2i| \le j\le n \\ n\equiv j \bmod 2}} \Big( \Big( 1- \frac 1p\Big) (G(i) + G(n-i))   + \frac{G(j)}{p} \Big)
$$ 
Given $i$, the minimum over suitable $j$ of $G(j)$ is clearly $\le G(n) \le G(i) + G(n-i)$ by Proposition \ref{prop3.1}.   Therefore we 
see that the right side in the display above is smallest when $p=p_{r+1}$.   We conclude that 
$$ 
G(n) \ge \min\Big( G_r(n), \min_{1\le i < n} \min_{\substack{ |n-2i| \le j\le n \\ n\equiv j \bmod 2}} \Big( \Big( 1- \frac 1{p_{r+1}}\Big) (G(i) + G(n-i))   + \frac{G(j)}{p_{r+1}} \Big)\Big), 
$$
from which \eqref{3.3} follows.  
\end{proof} 

Proposition \ref{prop3.2} sets up a simple recursive procedure to obtain lower bounds for $G_r(n)$ and $G(n)$.  For example, we 
can easily compute $G_1(n)$ exactly.

\begin{lemma} \label{lem3.3}  We have 
\begin{equation} 
\label{3.4} 
G_1(n) = \frac{1}{12} \Big( \frac n3 + \frac 23 \Big( 1- \frac 12 + \frac 14 -\ldots + \frac{(-1)^{n-1}}{2^{n-1}}\Big)\Big),  
\end{equation} 
and this norm is attained for the list $[(-2)^j: \ 0\le j <n ]$.  
\end{lemma} 
\begin{proof}  We use induction on $n$; the case $n=1$ is trivial since $N([a]) = 1/12$ for any non-zero $a$.  
Temporarily we define $h(n)$ to be the right hand side of \eqref{3.4}, and we 
note that $h$ is monotone increasing in $n$.   

Applying \eqref{3.2} with $r=1$, and since $h$ is monotone, we obtain 
$$ 
G_1(n) \ge \min_{1 \le i <n} \Big( \frac{n^2}{12}, h(i) +\frac{(n-i)^2}{12}, \frac 12 \Big( h(i) + \frac{(n-i)^2}{12} + h(|n-2i|)\Big)\Big). 
$$ 
 Now it is easy to check that $i^2/12 + h(n-i) \ge h(n)$ for all $1\le i<n$, and that $i^2/12+h(n-i) + h(|n-2i|) \ge 2h(n)$,
  which establishes that $G_1(n) \ge h(n)$.    Direct calculation using \eqref{2.1} shows that equality is attained here for the list $[ (-2)^j: 0\le j <n]$.  
 \end{proof} 
 
 Further one can obtain asymptotically sharp lower bounds for $G(n)$ from Proposition \ref{3.2}, although for small values of $n$ 
 we shall need more precise bounds.  

\begin{proposition} \label{prop3.4} For all $r\ge 0$ we have 
\begin{equation}
\label{3.5} 
G_r(n)  \ge \frac{n}{12} \prod_{j=1}^r \Big( \frac{p_j-1}{p_j+1} \Big).  
\end{equation} 
Moreover, if $2^m \le n <2^{m+1}$ then 
\begin{equation} \label{3.6} 
G(n) \ge \frac{n}{12} \prod_{j=1}^{m} \Big( \frac{p_j-1}{p_j+1}\Big). 
\end{equation} 
\end{proposition} 
\begin{proof}   We establish \eqref{3.5} by induction on $r$ and $n$.  When $r=0$ the result holds for all $n$ as $G_0(n) =n^2/12$, and the result is also 
easy when $n=1$ for all values of $r$.   When $r=1$ the result follows from Lemma \ref{lem3.3}.      Therefore by induction hypothesis and \eqref{3.2} we obtain 
$$ 
G_r(n) \prod_{j\le r}\Big(\frac{p_j+1}{p_j-1}\Big) \ge 
\min_{1\le i <n} \min_{\substack{|n-2i| \le j< n\\ j\equiv n\bmod 2}} \Big( \frac{n}{12},  \Big(1-\frac 1{p_r}\Big) \frac{i}{12} +\Big(1+\frac 1{p_r}\Big) \frac{n-i}{12}  + \frac{1}{p_r}  \frac{j}{12}\Big) 
= \frac{n}{12},  
 $$
 which gives \eqref{3.5}. 
 
 The proof of \eqref{3.6} is similar, by induction on $n$.  Suppose $2^m \le n < 2^{m+1}$.   We apply \eqref{3.3} with $r={m-1}$ there.  
 Using the bound just established \eqref{3.5} and the induction hypothesis we find 
 $$ 
 G(n) \prod_{j=1}^{m} \Big( \frac{p_j+1}{p_j-1}\Big) \ge \min_{1\le i < n}  \Big( \frac{n}{12}, \Big(1-\frac{1}{p_m}\Big) 
 \prod_{j=1}^{m} \Big( \frac{p_{j}+1}{p_j-1}\Big) (G(i)+G(n-i)) + \frac{|n-2i|}{12p_m}\Big).
 $$
Now by symmetry we may assume that $i\le n/2 <2^m$ above, in which case the induction hypothesis gives the 
stronger bound $G(i) \ge \frac{i}{12} \prod_{j=1}^{m-1} (\frac{p_j-1}{p_j+1})$.   Thus we obtain 
$$ 
G(n) \prod_{j=1}^{m} \Big( \frac{p_j+1}{p_j-1}\Big) \ge \min_{1\le i \le n/2}  \Big( \frac{n}{12}, \Big(1+ \frac 1{p_m}\Big) \frac{i}{12} + \Big(1-\frac{1}{p_m}\Big)\frac{n-i}{12} + \frac{n-2i}{12 p_m}\Big) 
= \frac{n}{12}. 
$$ 
This completes the proof of the proposition. 
\end{proof}

\begin{corollary} \label{cor3.5}  For $2\le n \le 11$ the following table gives lower bounds for $G(n)$ and $G(n;1)$:  
\begin{center} 
\begin{tabular} {|c || c | c| c| c| c| c| c| c| c| c| }
\hline
$n$ & $2$ & $3$ & $4$ & $5$ & $6$ & $7$ & $8$ & $9$ & $10$ & $11$  \\
\hline 
$G(n)\ge $ &${1}/{12}$ & $ 1/8 $ & $ 1/9 $ & $ 1/6$  & $17/108$ & $5/27$ & $37/216$ & $95/432$ & $2/9$ & $325/1296$ \\ 
\hline 
$G(n;1)\ge $ & $1/6$ & $1/6$ &$ 1/6$ & $7/{36}$ & ${7}/{36}$ & ${17}/{72}$& $ 2/9$ & $ {55}/{216}$ &${55}/{216}$ & ${8}/{27}$ \\ 
\hline
\end{tabular} 
\end{center} 
If $n\ge 11$ is odd then $G(n) >1/4$, and if $n\ge 82$ then $G(n)>1$. 
\end{corollary}  
\begin{proof} From Proposition \ref{prop3.4} it follows that if $2^m \le n <2^{m+1}$ then 
$$ 
G(n) \ge \frac{2^m}{12} \prod_{j=1}^{m} \Big( \frac{p_j-1}{p_j+1} \Big). 
$$ 
The right side above is increasing in $m$ for $m\ge 2$, and a small calculation shows that it exceeds $1$ for $m=8$.  
It follows that $G(n) >1$ for all $n\ge 2^8=256$.  

In the range $n\le 256$, we used Proposition \ref{prop3.2} to compute lower bounds for $G_1(n)$, $G_2(n)$ and $G_3(n)$, 
and then used \eqref{3.3} there (with $r=3$) to compute a lower bound for $G(n)$.   Once lower bounds for $G(n)$ have been computed, 
one obtains bounds for $G(n;1)$ using Proposition \ref{prop3.1}.   The values of $G(n)$ and $G(n;1)$ for $2\le n\le 11$ displayed above were 
extracted from this table.  From the table, one readily finds that $G(n)>1$ for $n\ge 82$, and $G(n)>1/4$ for odd $n\ge 11$ (and this also holds 
for even $n\ge 14$).   
\end{proof} 

For $2\le n\le 6$, the bounds for $G(n)$ given in Corollary \ref{cor3.5} are tight, and for $2\le n\le 8$ the values of $G(n;1)$ are exact.  
In Section 9, we shall establish that $G(7)=5/24$ and that $G(8) = 8/45$.  
For large $n$, asymptotically \eqref{3.6} furnishes the lower bound 
\begin{equation} 
\label{3.7} 
G(n) \ge \frac{n}{12} \prod_{j\le \lfloor \log n/\log 2 \rfloor}  \Big( \frac{p_{j}-1}{p_j+1} \Big) 
\sim \frac{n}{12} \Big(\frac{\pi^2 e^{-2\gamma}}{6 (\log \log n)^2}\Big),  
\end{equation} 
since, by Mertens's theorem,  
$$ 
\prod_{p\le x} \Big( \frac{p-1}{p+1} \Big) = \prod_{p\le x} \Big(1 -\frac{1}{p}\Big)^2 \Big(1- \frac 1{p^2}\Big)^{-1} \sim 
\frac{\pi^2}{6} \frac{e^{-2\gamma}}{(\log x)^2}. 
$$ 
In Section 12 we show that this bound is attained asymptotically.

\section{Some Observations and understanding norms for small lengths} 


\noindent In this section we describe an involution on lists which preserves the norm, and 
then compute explicitly the lists of small norm for lengths $2$, $3$, and $4$.  These calculations  
will be used in Section 6 to classify all the factorial ratios with $D=1$ and $K+L =5$.  

\subsection{An involution that preserves norms}   

Given a (non-degenerate) list $\frak a$, we define a new list $\overline{\frak a}$ as follows.  If $a$ is an 
even element in $\frak a$, leave it unaltered in $\overline{\frak a}$.  If $a$ is an odd element in $\frak a$, replace if 
with two elements $2a$, $-a$ in $\overline{\frak a}$.  After performing this operation on all elements in $\frak a$, 
remove any degeneracies from the list $\overline{\frak a}$.  For example, if $\frak a =[30, 1, -10, -15,-6]$ then performing the procedure we 
obtain $[30, 2, -1, -10, -30, 15, -6]$ and removing degeneracies we end up with $\overline{\frak a} = [15, 2, -10, -6,-1]$.  

The relation between  $\overline{\frak a}$ to $\frak a$ becomes clear when one considers the associated periodic functions 
$\frak a(x)$ and $\overline{\frak a}(x)$.  Since $\psi(2x) = \psi(x) +\psi(x+1/2)$ one sees that $\overline{\frak a}(x) = \frak a(x+1/2)$, 
which explains why this operation is an involution, and why the norm is preserved.  

This involution will help in seeing why several lists below have the same norm.  We point out a pleasant exercise to the reader: 
if $\frak a$ is a list corresponding to an integral factorial ratio with $D=1$, then $\overline{\frak a}$ or $-\overline{\frak a}$ will 
also yield a list corresponding to an integral factorial ratio with $D=1$.

\subsection{Length $2$}  If $[a,b]$ is a primitive list of length $2$ then a small calculation gives 
\begin{equation} 
\label{4.1}
N([a,b]) = \frac 1{12} + \frac{1}{12} +\frac{2}{12 ab} = \frac{1}{6} \Big( 1 + \frac{1}{ab}\Big). 
\end{equation}
Note here that the list $[1, -1]$ is excluded as it is degenerate.  Thus the smallest value of the 
norm is attained for the list $[1,-2]$ and so $G(2) =1/12$.   Also, from our formula it is clear that 
$G(2;1)= 1/6$, and this also follows from Proposition \ref{prop3.1}.

\subsection{Length $3$}   Our goal is to classify all primitive lists $\frak a$ of length $3$ with norms below $43/216$.

\begin{lemma}  \label{lem4.1}  If $\frak a$ is a primitive list of length $3$ with $N(\frak a) < 43/216$ then $\frak a$ is 
of the form $[a, -ka, b]$ with $2\le k \le 5$ and $(a,b)=1$. 
\end{lemma} 
\begin{proof}  Suppose $\frak a$ does not contain a pair of elements $a$, $-ka$ with $2\le k \le 5$.  If all three elements of $\frak a$ have the same 
sign, then clearly $N(\frak a) \ge 1/4$.   Without loss of generality we 
can now assume that $\frak a = [a, -b, c]$ with $a$, $b$, $c$ all positive.  From the assumption on $\frak a$ we have that 
$(a,b)^2/(ab) \le 1/6$ and $(b,c)^2/(bc) \le 1/6$, and note that $(a,c)^2/(ac) \ge (a,b)^2 (b,c)^2/(ab b c)$.   From these observations 
we see easily that 
$$ 
N([a,-b,c]) = \frac 1{12} \Big( 3 - \frac{2(a,b)^2}{ab} - \frac{2(b,c)^2}{bc} + \frac{2(a,c)^2}{ac} \Big) \ge\frac 1{12} \Big( 3- \frac 26-\frac 26 +\frac 2{36}\Big) 
= \frac{43}{216}. 
$$ 
This is a sharp bound, attained by $\frak a = [4,-6,9]$. 
\end{proof} 
 
Suppose that $\frak a$ contains two elements of the form $a$, $-pa$ for some prime $p$.  Thus $\frak a = [a, -pa, b]$ with $(a,b)=1$.  
Here we can compute easily that 
\begin{equation} 
\label{4.2} 
N([a,-pa,b]) = \begin{cases} 
\frac{1}{4} - \frac 1{6p} + \frac{p-1}{6pab} &\text{if } p\nmid b \\ 
\frac 14 - \frac{1}{6p} - \frac{p-1}{6ab} &\text{if } p|b.\\
\end{cases}
\end{equation}

Finally, suppose $\frak a$ contains two elements $a$, $-4a$, and the third element is $b$ coprime to $a$.  Here we can compute that 
$$
N([a,-4a,b]) = \begin{cases} 
\frac{5}{24} + \frac{1}{8ab} &\text{if } 2\nmid b\\ 
\frac 5{24} &\text{if } 2 \Vert b\\ 
\frac{5}{24} - \frac{1}{2ab} &\text{if } 4|b. \\
\end{cases} 
$$ 

We give some examples of small norms from these families.   From the family $[a,-2a,b]$ (which is the family of Type A lists): 
$$ 
{\text{ Norm }} \tfrac 18: [1,-2,4]; \ \ \text{Norm }   \tfrac{5}{36}:  [1,-2,-3],\  [2,3,-6],\  [1,-3,6], \ [1,-2,6]);
$$
$$
\text{Norm } \tfrac{7}{48}: [1,-4,8], \ [1,-2,8]; \ \ \text{Norm } \tfrac{3}{20}: 
[1,-2,-5],\   [1,-5,10],\  [1,-2,10],\  [2,5,-10].
$$ 
This family contains the list with smallest norm $G(3)=1/8$, and the limiting value in this family is $1/6$ 
which equals $G(3;1)$.  

From the family $[a,-3a,b]$ (but not included in the previous family), we have 
$$ 
\text{Norm } \tfrac{17}{108}:  [1,-3,9]; \ \ \text{Norm } \tfrac 16: [1,-3,-4], \ [3,4,-12];
$$
$$
\text{Norm } \tfrac{31}{180}:  [1,-3,-5], \ [3,5,-15], \ [1,-3,15], \ [1,-5,15]; \ \ \text{Norm } \tfrac{19}{108}:  [2,-3,9],\  [2, -6,9]; 
$$
$$  
\text{Norm } \tfrac{5}{28}: [1,-3,-7], \ [1,-7,21], \ [1,-3,21], \ [3,7,-21]. 
$$ 

Finally, a couple of small examples from the family $[a,-4a,b]$ (but not included in the previous two families): 
$$ 
\text{Norm } \tfrac{17}{96}: [1,-4,16]; \ \ \text{Norm } \tfrac{11}{60}: \ [1,-4,-5], \ [4,5,-20]. 
$$




\subsection{Length $4$}   First we deal with the Type A lists; these are non-degenerate lists $\frak a$ of the form $[a, -2a, b, -2b]$ with $(a,b)=1$.  
A small calculation gives 
\begin{equation} 
\label{4.3} 
N([a,-2a, b, -2b]) = \begin{cases} 
\frac 16 + \frac{1}{6ab} &\text{ if } 2\nmid ab\\ 
\frac 16-\frac{1}{12ab} &\text{ if } 2|ab.\\ 
\end{cases}
\end{equation}
Note the similarity with norms of the two term list $[a,b]$ (see \eqref{4.1}) and the three term list $[a,-2a, b]$ (see \eqref{4.2} with $p=2$ there), which is explained by the involution described in Section 4.1.   
The smallest norm among these lists is attained for $[1,-2,-3,6]$ which has norm $1/9$, and in fact $G(4)=1/9$.  Further the norms in this family have a limit point $1/6$, 
and in fact $G(4;1) =1/6$.  

The next lemma determines all the Type B lists of length $4$  with norm below $11/60$. 

\begin{lemma} \label{lem4.2}  Suppose $\frak a$ is a primitive list of length $4$ not of the form $[a,-2a,b,-2b]$.  Then $N(\frak a) \ge 11/60$ unless 
$\frak a$ is one of the lists given below.   
$$ 
\text{Norm } \tfrac 16: \ \ [1,-3,6,-12], \  [1,-3,-4,6], \ [1,-3,-4,12], \  [1,-2, 4,-12], \  [1,-2,-3,4],
$$
$$ 
\ \ \ \ [1,-2,-3,12], \ [1,-4,-6,12], \ [2,-3,-4,12], \  [3,-4,-6,12].  
$$ 
$$ 
\text{Norm } \tfrac{19}{108}: \ \ [1,-3,9,-18], \ [1,-2,6,-18], \  [1,-2,-3,9], \ [2,-6,-9,18]. 
$$ 
$$
\text{Norm } \tfrac{17}{96}: [1,-2,4,-16], [1,-4,8,-16]. \qquad \text{Norm } \tfrac{8}{45}: [1,-3,-5,15].
$$
$$
\text{Norm } \tfrac{13}{72}: [1,-3,-4,8], \ [1,-3,12,-24], \  [1,-2,8,-24], \ [3,-6,-8,24].
$$ 
\end{lemma}

\begin{proof} Suppose first that $\frak a$ is a primitive list of Type B which is at most $4$ separated.  Then from Proposition \ref{prop2.4} we know that 
the elements of $\frak a$ must be divisors of $4^3 \times 3^3 =1728$.   A small computer calculation shows that there are 
exactly nineteen such lists with norm below $11/60$; these account for all the lists given in the lemma with the exception of the list with norm $8/45$.

Suppose now that $\frak a$ is $k$--separated with $k\ge 5$, and let $\frak b$, $\frak c$, $\widetilde{ \frak b}$, $\widetilde{\frak c}$ have 
the meanings of Definition \ref{def2.1}.  The argument that follows will recur several times in our subsequent work.   
There are  two cases: either $\ell(\frak b) =1$ and $\ell(\frak c)=3$, or  $\ell(\frak b) = \ell(\frak c) =2$.

Consider the former case first.   Here $N(\frak b) = 1/12$, and $\widetilde{\frak b}+\widetilde {\frak c}$ is non-empty and therefore 
has norm at least $1/12$.   Since every list of length $3$ has norm at least $1/8$, we have $N(\frak c) \ge 1/8$, and we conclude from Proposition 
\ref{prop2.2} that 
$$
N(\frak a) \ge \Big(1- \frac 1k\Big) \Big( \frac{1}{12} + \frac{1}{8}\Big) + \frac{1}{k} \frac{1}{12} \ge \frac {11}{60}.
$$

Now consider the case where $\frak b$ and $\frak c$ are both of length $2$.   Here we distinguish two further possibilities: either $\frak b= \frak c$ or $\frak b \neq \frak c$.  
Consider the first possibility.   Since $\frak a$ is given to be of Type B, the case $\frak b=\frak c = [1,-2]$ is ruled out.  If $\frak b= \frak c= [1,-3]$ then 
$\frak a$ is of the form $[a,-3a, b, -3b]$ a simple calculation shows that 
$$ 
N([a,-3a,b,-3b]) = \begin{cases} 
\frac{2}{9} - \frac{2}{9ab} &\text{if } 3|ab\\ 
\frac 29+\frac{2}{9ab} &\text{if  }3\nmid ab. \\
\end{cases} 
$$
This produces one new example with norm below $11/60$; namely the list $[1,-3,-5,15]$ with norm $8/45$.  
 If $\frak b=\frak c$ is not $[1,-2]$ or $[1,-3]$ then its norm is at least $1/8$, 
 and by Proposition \ref{prop2.2} we have $N(\frak a) \ge \frac 45 (\frac 18 +\frac 18) =\frac 15$.

Finally suppose $\frak b \neq \frak c$,  in which case $N(\widetilde{\frak b}+\widetilde{\frak c}) \ge 1/12$.      If now $N(\frak b) + N(\frak c) 
\ge 5/24$ then it follows that $N(\frak a) \ge \frac 45 \frac 5{24} +\frac 1{60} = \frac{11}{60}$.   On the other hand, if $N(\frak b) + N(\frak c)< 5/24$, 
then we must have one of $\frak b$ or $\frak c$ being $[1,-2]$ and the other being $[1,-3]$, so that $\frak a$ is of the form $[a,-2a, b,-3b]$.  
Computing this norm, we can check that there are no new examples of lists with norm below $11/60$.  
 \end{proof}


  \section{Integral factorial ratios for $D=1$:  Toward the proof of Theorem \ref{thm1.1}}  
  

\noindent  In this section, we prepare the foundations for the proof of Theorem \ref{thm1.1}.   Here $D=L-K=1$, and we are looking for primitive 
tuples $(a_1, \ldots, a_K, b_1, \ldots, b_{K+1})$ that lead to integral factorial ratios.  Recall from the introduction that we can associate to such a tuple 
the primitive list $\frak a = [a_1, \ldots, a_K, -b_1, \ldots, -b_{K+1}]$ which has odd length $2K+1$, and has sum  $s(\frak a)= 0$.  We remarked in the introduction 
that $\frak a(x)$ takes the values $-1/2$, $1/2$ (away from finitely many points) so that $N(\frak a) = 1/4$.

In fact, if $\frak a$ is any primitive list of odd length and with sum $s(\frak a) =0$, then the values of $\frak a(x)$ may be seen to be in ${\Bbb Z}+1/2$, 
so that $N(\frak a) \ge 1/4$.   If $N(\frak a) =1/4$, then it follows that $\frak a(x) = \pm 1/2$ (apart from finitely many points).  Further, if $\frak a$ has 
$K$ positive entries and $L$ negative entries, then for suitably small but positive $x$, one has $\frak a(x) = (K-L)/2$, so that one necessarily has $|L-K|=1$.  
By flipping the sign of $\frak a$ if necessary, we see that primitive lists of odd length $2K+1$, with sum $0$, and norm $1/4$ correspond exactly to integral factorial ratios 
with $K$ factors in the numerator, and $K+1$ factors in the denominator.

Thus from now on we focus on the equivalent problem of determining all primitive lists $\frak a$ with $\ell(\frak a)$ odd, $s(\frak a) =0$ and $N(\frak a) = \frac 14$.  
If $\ell(\frak a) =3$, then the condition $s(\frak a)=0$ means that $\frak a$ is of the form $[a+b, -a, -b]$ for coprime positive integers $a$ and $b$.  This list has 
norm $\tfrac 14$ and corresponds to the integrality of the binomial coefficients $\binom{(a+b)n}{an}$.   If $\ell(\frak a) \ge 11$ is odd, then by Corollary \ref{cor3.5}, 
$N(\frak a) > \frac  14$, and so there are no integral factorial ratios with $D=1$ and $K\ge 5$.  Thus we are left with the cases $\ell(\frak a) = 5$, $7$, $9$, and 
in Sections 6, 8, and 10 we shall classify all the integral factorial ratios with $D=1$ and these values for $2K+1$.  These results (which will fully establish Theorem 
\ref{thm1.1}) will rely on an understanding of lists with small norm (described in Sections 4, 7, and 9) obtained by combining our ideas in Sections 2 and 3 together with 
some computer calculations.   In this section we give a quick proof of a qualitative version of Theorem \ref{thm1.1}, showing that there are only finitely many sporadic 
examples.  

\begin{theorem} \label{thm5.1} In the case $D=1$, apart from the three infinite families given in Theorem \ref{thm1.1} there are only finitely many 
sporadic examples of primitive integral factorial ratios.  
\end{theorem}  

To prove Theorem \ref{thm5.1}, and for our subsequent work toward Theorem \ref{thm1.1}, we require the following lemma, which 
adds to Definition \ref{def2.1} in the situation where $\ell(\frak a)$ is odd and $s(\frak a)=0$.  
\begin{lemma} \label{lem5.2}  Suppose $\frak a$ is a primitive list with $s(\frak a)=0$, $\ell(\frak a)$ odd, and $N(\frak a) \le \frac 14$.   
Suppose $\frak a$ is $k$--separated, with $\frak b$, $\frak c$, $B$, $C$, $\widetilde{\frak b}$ and $\widetilde{\frak c}$ as in Definition \ref{def2.1}.  
Then $s(\frak b)$ and $s(\frak c)$ are non-zero, and one has 
$$ 
\pm B = -\frac{s({\frak c})}{(s({\frak b}),s({\frak c}))}, \qquad \pm C = \frac{s({\frak b})}{(s({\frak b}),s({\frak c}))}.
$$
\end{lemma} 
\begin{proof}  Since $s(\frak a) = Bs(\frak b) + C s(\frak c) =0$ and $\frak a$ is primitive, if we knew that $s(\frak b)$ and $s(\frak c)$ are non-zero, then 
the values of $B$ and $C$ would be specified as stated in the lemma.  Further, note that if one of $s(\frak b)$ or $s(\frak c)$ is 
zero, then the other must also be zero.  

Suppose then that $s(\frak b) = s(\frak c) =0$.  Since $\ell(\frak a) = \ell(\frak b) + \ell(\frak c)$ is odd, one of $\ell(\frak b)$ or $\ell(\frak c)$ must be odd and 
the other even.  As remarked above, a list of odd length with sum zero has norm at least $1/4$.  It follows that $N(\frak b) + N(\frak c) \ge \frac 14 + \frac 1{12}$.  
Further $\widetilde{\frak b} + \widetilde{\frak c}$ is also a list of odd length with sum zero, and thus has norm at least $\frac 14$.  By Proposition \ref{prop2.2} we 
conclude that 
$$ 
N(\frak a) = \Big( 1- \frac 1k \Big) \big( N(\frak b) + N(\frak c) \big) + \frac 1k N (\widetilde{\frak b} + \widetilde{\frak c}) 
\ge \Big( 1-\frac 1k\Big) \Big( \frac 14 +\frac{1}{12} \Big) +\frac{1}{k} \frac{1}{4} > \frac 14,
$$ 
which contradicts our assumption that $N(\frak a) \le \frac 14$.   
\end{proof}

\begin{proof}[Proof of Theorem \ref{thm5.1}]  We are interested in primitive lists $\frak a$ with $s(\frak a)=0$, $\ell(\frak a) =5$, $7$, or $9$, and $N(\frak a) =\frac 14$.  
We may assume that $\frak a$ is $k$--separated with $k$ sufficiently large, since by Proposition \ref{prop2.4} there are only finitely many lists with specified 
length and bounded separation.  Given that $\frak a$ is $k$--separated, let $\frak b$ and $\frak c$ have the meanings of Definition \ref{def2.1}.  Given a sufficiently small 
positive $\epsilon$, by ensuring that $k$ is large enough, we may assume that $N(\frak b) + N(\frak c) \le \tfrac 14+\epsilon$.  If $\frak b$ and $\frak c$ are 
specified primitive lists, then Lemma \ref{lem5.2} shows that there is at most one way to combine these to give $\frak a$.

Suppose now that $\ell(\frak a)=5$.  The Type A lists $[a, -2a, b, -2b, (a+b)]$ account for the infinite families of length $5$ given in Theorem \ref{thm1.1}, and 
we now suppose that $\frak a$ is of Type B.   Then at least one of $\frak b$ or $\frak c$ must also be of Type B.  There now arise two possibilities:  $\ell(\frak b)=1$ and 
$\ell(\frak c) =4$, or $\ell(\frak b)=2$ and $\ell(\frak c)=3$.   

If $\ell(\frak b) =1$ (so $\frak b =[1]$) then 
$\frak c$ must be a Type B list of length $4$, and by Lemma \ref{lem4.2} there are only finitely many such $\frak c$ with norm below $11/60$.  Thus in this 
case there are only finitely many choices for $\frak b$ and $\frak c$ with $N(\frak b) + N(\frak c) \le 1/12 +11/60=1/4+1/60$, whence there are only finitely many $\frak a$ 
arising from this case.   

Now consider the possibility $\ell(\frak b) =2$ and $\ell(\frak c)=3$.  If $\frak b=[1,-2]$, then $\frak c$ must be of Type B, and from our 
work in Section 4.3, there are only finitely many Type B lists $\frak c$ with norm below $31/180$.  Thus in this situation also there are only finitely many 
choices for $\frak b$ and $\frak c$ with $N(\frak b)+N(\frak c) \le \frac 14+\epsilon$, and hence for $\frak a$.  Finally suppose $\frak b \neq [1,-2]$ so that $N(\frak b) \ge 1/9$.  From our work in Section 4.3, 
there are only finitely many lists $\frak c$ with norm below $1/6-\epsilon$ for any $\epsilon >0$.  Once again this implies that there are only finitely many choices 
for $\frak b$ and $\frak c$ with $N(\frak b) + N(\frak c) \le \frac 14+\epsilon$, which completes our proof for the case $\ell(\frak a) =5$.

We now turn to the situation $\ell(\frak a) =7$, where there are three possibilities $\ell(\frak b)=1$, $\ell(\frak c)=6$; or $\ell(\frak b)=2$, $\ell(\frak c)=5$; or 
$\ell(\frak b)=3$, $\ell(\frak c)=4$.  In the first case, we have $\frak b=[1]$, and by Corollary \ref{cor3.5} we know that $G(6;1) \ge 7/{36}$ which 
means that there are only finitely many $\frak c$ with $N(\frak c) \le 7/36 -\epsilon$.  But this means that there are only finitely many 
choices for $\frak b$ and $\frak c$ with $N(\frak b) + N(\frak c) \le 1/12+ 7/36-\epsilon = \frac 14 +\frac{1}{36}-\epsilon$.  

Now take the second case $\ell(\frak b)=2$ and $\ell(\frak c)=5$.  If $\frak b=[1,-2]$ with norm $1/12$, then again by Corollary \ref{cor3.5} we know that $G(5;1) \ge 7/36$, 
so that again there are only finitely many choices for $\frak c$ with $N(\frak b) + N(\frak c) \le 1/12 +7/36 -\epsilon$.  If $\frak b \neq [1,-2]$ then $N(\frak b)\ge 1/9$, and 
$N(\frak c) \ge 1/6$ by Corollary \ref{cor3.5}.  So in this case $N(\frak b) + N(\frak c) \ge 1/9+1/6 =\frac 14+\frac{1}{36}$, and again we are done.  

In the third case $\ell(\frak b) =3$ and $\ell(\frak c) =4$, since $N(\frak b) \ge 1/8$ we must have $N(\frak c) \le 1/8+\epsilon$.  Similarly, since $N(\frak c) \ge 1/9$ 
we must have $N(\frak b) \le 5/36+\epsilon$.  Now by Corollary \ref{cor3.5}, or by our more precise work in Sections 4.3 and 4.4, we know that 
there are only finitely many $\frak b$ with $1/8 \le N(\frak b) \le 5/36+\epsilon$, and finitely many $\frak c$ with $1/9 \le N(\frak c)\le 1/8+\epsilon$.  This 
finishes the proof for $\ell(\frak a) =7$.  

Lastly, suppose $\ell(\frak a) =9$.  Here we can even ignore the condition that $s(\frak a)=0$, because there are only finitely many primitive lists of 
length $9$ with norm at most $\tfrac 14$.  This follows from the table in Corollary \ref{cor3.5}, which reveals that $G(9;1)\ge 55/216 > \frac 14$.  
 \end{proof}

\section{Classifying integral factorial ratios of length $5$}  


\subsection{Finding the examples}   Our goal is to find all primitive lists $\frak a$ with  length $5$, with $s(\frak a)=0$ and $N(\frak a)=\frac 14$.  
 Recall that there are infinite families of examples arising from 
$$ 
\frak a =[a, -2a, b, -2b, a+b]
$$ 
where $a$ and $b$ are integers with $(a,b)=1$ (and omitting degenerate examples).  
We now find the twenty nine sporadic examples of integral factorial ratios of length $5$, and then demonstrate 
in the next two subsections that there are no further examples.  
 
Now consider the two parameter family 
$$ 
\frak a= [a,-2a, b, -3b, c=a+2b], \qquad \text{with } (a,b)=1. 
$$ 
Direct calculation gives, making use of $(a,b)= (b,c)=1$ and $(a,c) = (a,2)$ and so on,  that $N(\frak a)$ equals 
$$ 
\frac 14 + \frac{1}{12} \Big(\frac 13+ \frac{2}{ab} \Big( 1 - \frac{(a,3)^2}{3} - \frac{(2,b)^2}{2} +\frac{(2,b)^2(3,a)^2}{6}\Big) 
+\frac{2}{ac} \Big( (a,2)^2 -\frac{(2a,c)^2}{2}\Big) + \frac 2{bc} \Big( 1- \frac{(3,c)^2}{3}\Big) \Big). 
$$    
A little calculation, considering the possible values of $(a,3)$, $(b,2)$ etc, gives the lower bound 
$$ 
N(\frak a) \ge \frac 14 + \frac{1}{12} \Big( \frac 13 -\frac{4}{|ab|} - \frac{4}{|bc|} - \frac{8}{|ac|}\Big), 
$$ 
so that in order to get an integral factorial ratio, one must have $|ab| \le 36$, or $|bc|\le 36$, or $|ac|\le 72$.  Since $|a|$, $|b|$, and $|c|$ 
are all integers at least $1$, the condition $|ab| \le 36$ implies that $|a|$ and $|b|$ are at most $36$.  Similarly $|bc| \le 36$ implies 
$|b| \le 36$ and $|c| = |a+2b| \le 36$ which together imply $|a| \le 108$.  Finally $|ac|\le 72$ implies that $|a|$ and $|c|$ are below $72$, 
and so $|b| \le (|a|+|c|)/2 \le 72$.  Thus in all cases we see that $|a|\le 108$ and $|b|\le 72$.   It is a simple matter to check these cases on 
a computer, and leads to the following nineteen examples: 
$$
[1,15, -2,-5,-9]; \qquad [1, 9, -2,-3,-5]; \qquad [2,12,-1,-4,-9]; \qquad [4,9,-2,-3,-8]
$$
$$
[4, 15,-2, -5,-12]; \qquad [3, 12, -4, -5, -6]; \qquad [4, 15, -5, -6, -8]; \qquad   [3, 20, -1, -10, -12]
$$ 
$$ 
[5, 12, -10, -4, -3]; \qquad [6, 10, -2, -5,-9]; \qquad [1,12,-3,-4,-6]; \qquad [3, 12, -1, -6,-8] 
$$ 
$$ 
[7,15,-3,-5,-14]; \qquad [3,14, -1, -7, -9]; \qquad [1, 14,-3,-5,-7]; \qquad [4, 18, -1, -9, -12]
$$ 
$$
[2, 18, -5, -6, -9]; \qquad [3, 20, -4, -9, -10]; \qquad [1, 20, -3, -8, -10]. 
$$

Next we considered all lists $\frak{a} = [a,b,c,d, e=-(a+b+c+d)]$ where four of the elements (say, $a$, $b$, $c$, and $d$) are 
divisors of $2^6 \times 3^3 \times 5^3$.  This gives the remaining ten examples: 
$$ 
[1,12, -2, -3,-8], \qquad  [2,12, -3, -4, -7], \qquad [1, 18, -4,- 6, -9], \qquad [1, 20, -4, -7, -10], 
$$
$$ 
 [2, 15, -3, -4, -10], \qquad  [1, 24,-5, -8,-12], \qquad [2, 15,-1, -6, -10], \qquad [1,15,-3, -5,-8], 
$$ 
$$
[1,30, -6, -10, -15]; \qquad [2,9,-1,-4,-6].
$$ 
  
Note that our check above includes all lists $\frak a$ that are at most $6$--separated.  We may see this as in Proposition \ref{prop2.4} (and see the remarks following it), and noting 
further that at least two of the elements of $\frak a$ must be indivisible by any prime because the sum of the elements of $\frak a$ is zero.  
Indeed a list that is at most $6$--separated must have all elements being divisors of $2^6 \times 3^3 \times 5^3$, and in our check we did not 
insist that $e$ must also be a divisor of $2^6 \times 3^3 \times 5^3$.

 To summarize, we may assume that $\frak a$ is $k \ge 7$ separated, and that $\frak a$ is not of the form $[a,-2a,b,-2b,a+b]$ or of the form $[a,-2a,b,-3b,a+2b]$, and we 
 wish to show that there are no further examples apart from the twenty nine mentioned above.   Below $\frak b$, $\frak c$, $\widetilde{\frak b}$ and $\widetilde{ \frak c}$ have the 
 meanings assigned in Definition \ref{def2.1}. 
 
\subsection{The case $\ell(\frak b)=2$ and $\ell(\frak c)=3$}   Since $\frak a$ is at least $7$--separated, we see that 
$$
\tfrac 14 =N(\frak a) \ge \tfrac 67 (N(\frak b) + N(\frak c)) + \tfrac{1}{84},
$$ 
so that 
\begin{equation} 
\label{6.1} 
N(\frak b) + N( \frak c) \le \tfrac{5}{18}.
\end{equation} 

If $\frak b = [1,-2]$ then \eqref{6.1} gives $N(\frak c) \le \frac{7}{36}$, and Lemma \ref{lem4.1} shows that $\frak c$ must be 
of the form $[a,-\ell a, b]$ with $2\le \ell \le 5$.   Further, the possibilities $\ell =2$ and $3$ are excluded from our work in the 
Section 6.1 above.   Thus by our work in Section 4.3 there are only a small number of possibilities for $\frak c$;  for example, in the case 
$\ell=4$ we only need to consider lists $[a,-4a,b]$ with $|ab|\le 36$ when $4|b$, and with $|ab|\le 9$ when $b$ is odd, and there are no possibilities 
with $2 \Vert b$.    Checking these cases we obtain no new examples of five term factorial ratios.  

If $\frak b =[1,-3]$ (with norm $1/9$) then \eqref{6.1} gives $N(\frak c) \le \frac 16$.   The possibility that $\frak c$ is of the form $[a,-2a,b]$ 
is excluded, and therefore by our work in Section 4.3 there are only three possibilities for $\frak c$, which are easily checked to produce 
no new examples. 

Finally we may assume that $N(\frak b) \ge \tfrac 18$ which implies that $N(\frak c) \le \tfrac{11}{72}$.  Further, since $N(\frak c) \ge \frac 18$ we 
also have that $N(\frak b) \le \frac{11}{72}$.  Thus, by the work in Sections 4.3 and 4.4, 
 there are only a small number of possibilities for $\frak b$ and for $\frak c$, and these are easily checked.

\subsection{The case $\frak b =[1]$ and $\ell(\frak c)=4$}   From our work in Section 6.1, we may assume 
that not all the elements of $\frak c$ are divisors of $2^6 \times 3^3 \times 5^3$, so that the list $\frak c$ must be at least $7$--separated.  
Furthermore, since $\ell(\widetilde{\frak b} + \widetilde{\frak c}) \ge 3$ we must have 
$$ 
\tfrac 14 = N(\frak a) \ge \tfrac 67 (\tfrac{1}{12}  + N(\frak c)) + \tfrac{1}{56},
$$  
so that $N(\frak c) \le \frac{3}{16}$.   The next lemma shows that there are no lists $\frak c$ that are at least $7$--separated that are 
not of the form $[a,-2a, b,-2b]$ or $[a,-2a, b,-3b]$ and with norm at most $\frac 3{16}$, and this will finish our classification of the five term factorial ratios. 

\begin{lemma} \label{lem6.1}  If $\frak c$ is a primitive list of length $4$ which is at least $7$--separated and $\frak c$ is 
not of the form $[a,-2a, b,-2b]$ or $[a,-2a, b,-3b]$ then $N(\frak c) > \frac{3}{16}$.  
\end{lemma} 
\begin{proof}  If $\frak c$ splits into lists of length $1$ and $3$ then $N(\frak c) \ge \tfrac{6}{7} (\tfrac{1}{12} + \frac{1}{8}) + \frac{1}{84}>\frac 3{16}$.  

Suppose now that $\frak c$ splits into two lists of length $2$.  If these two sublists are the same, then they must each have norm at least $1/9$ (since $\frak c$ is 
not of the form $[a,-2a,b,-2b])$ and then $N(\frak c) \ge \frac 67 ( \frac 19 + \frac 19) > \frac{3}{16}$.  If these two sublists are different then their norms must 
add up to at least $\frac{1}{8}+\frac{1}{12}$ (since $\frak c$ is not of the form $[a,-2a, b,-3b]$) and one has $N(\frak c) \ge \frac 67 (\frac 18+\frac{1}{12}) + \frac{1}{84} > \frac{3}{16}$.  
\end{proof}

\section{Lists of length $5$ and $6$} 


\noindent We prepare for the work in the next section (classifying factorial ratios of length $7$) by identifying the lists of length 
$5$ and $6$ with small norm. 

\subsection{Length $5$}  


\begin{lemma} \label{lem7.1}  Suppose $\frak a$ is a Type A primitive list of length $5$; that is, $\frak a$ is of the form $[a, -2a, b, -2b, c]$.   Then $N(\frak a) > 31/168$ 
except for the following lists: 
$$ 
\text{Norm } \tfrac 16: \ \  [1, -2, -3, 4, 6],  \ [1, -2, -3, 6, -12], \ [ 2, 3, -4, -6, 12], \ [1, -2, 4, 6, -12]. 
$$
$$ 
\text{Norm } \tfrac{19}{108}: \  \ [1,-2,-3,6,9], \  [1,-2,-3,6,-18], \   [1,-3,6,9,-18], \ [2,3,-6,-9,18].
$$ 
$$ 
\text{Norm } \tfrac{17}{96}: \ \ [1,-2,4,-8,16]. 
$$ 
$$
\text{Norm } \tfrac{13}{72}: \ \ [1,-2,-3,6,8], \ [1,-2,-3,6,-24], \ [1,-4,8,12,-24], \  [3,4,-8,-12,24]. 
$$
$$ 
\text{Norm } \tfrac{11}{60}: \ \  [1,-2,4,10,-20], \  [1,-2,-5,10,-20], \ [1,-2,4,-5,10], \  [1,-2,-3,6,15], 
$$ 
$$ 
\hskip .6 in [1,-2,-3,6,10], \  [1,-2,-3,-5,6], \  [1, -2,-3,6,-30], \  [1, -5, 10,15,-30],
$$
$$
\hskip .6 in [2,-4,5,-10,20], \  [2,5,-10,-15,30], \  [3,5,-10,-15,30], \ [5,-6,-10,-15,30].
$$ 
\end{lemma} 
\begin{proof}    Suppose first that $\frak a$ is at most $7$ separated.   Then, arguing as in Proposition \ref{prop2.4},
 the elements of $\frak a$ must be divisors of $2^6 \times 3^2 \times 5^2 \times 7^2$, and a computer calculation produced 
 the catalogue of lists with norm at most $31/168$ given in the lemma.

Now suppose that $\frak a$ is $k$-separated with $k \ge 8$.  If it splits into lists $\frak b$ of length $1$ and $\frak c$ of length 
$4$ then (since $\ell(\widetilde{\frak b}+\widetilde{\frak c})$ is odd and at least $3$, and so $N(\widetilde{\frak b} +\widetilde{\frak c})\ge \frac 18$) 
$$ 
N(\frak a) \ge \frac{7}{8} \Big( \frac {1}{12} +\frac{1}{9}\Big) + \frac 18 \frac{1}{8} = \frac{107}{576}. 
$$ 
If $\frak a$ splits into lists of length $2$ and $3$ then 
$$ 
N(\frak a) \ge \frac{7}{8} \Big( \frac{1}{12} + \frac {1}{8} \Big) + \frac{1}{8} \frac{1}{12} =\frac{37}{192}. 
$$ 
Both these values are $>31/168$, completing our proof.  \end{proof} 

\begin{lemma} \label{lem7.2}  If $\frak a$ is a primitive Type B list of length $5$ (thus not of the form $[a,-2a,b,-2b,c]$) then 
$N(\frak a) \ge  5/24$. 
\end{lemma} 
\begin{proof}   First we computed primitive Type B lists of length $5$ that are at most 
$4$ separated and found that the smallest attained norm is $5/24$, and there are $8$ lists attaining that norm.  
Now suppose $\frak a$ is primitive of Type B, and $k$--separated with $k\ge 5$.   Here we will check that 
the norm is $>5/24$.    If $\frak b$ and $\frak c$ have their usual meanings, then note that (since $\ell(\frak a)= 5$ is odd, and 
$\frak a$ is of Type B) one of $\frak b$ or $\frak c$ must be of Type B.   As in earlier situations, our argument splits into cases 
depending on the lengths of $\frak b$ and $\frak c$.  

 If $\ell(\frak b)=1$ and $\ell(\frak c)= 4$, then $\frak b=[1]$ and $\frak c$ must be of Type B.  By Lemma \ref{lem4.2} we know that  $N(\frak c) \ge 1/6$, and so by 
 Proposition \ref{prop2.2} 
 $$ 
N(\frak a) \ge \frac{4}{5} \Big( \frac{1}{12} + \frac{1}{6} \Big) + \frac{1}{5} \frac{1}{8} > \frac{5}{24}. 
$$ 

We may now suppose that $\ell(\frak b)= 2$ and $\ell(\frak c)=3$ and that (at least) one of $\frak b$ or $\frak c$ is of Type B.   
If $N(\frak b) + N(\frak c) 
\ge 1/4$ then 
$$ 
N(\frak a) \ge \frac{4}{5} \frac{1}{4} + \frac{1}{5} \frac{1}{12} = \frac{13}{60} > \frac{5}{24}. 
$$ 
Thus we may suppose that $N(\frak b) + N(\frak c) < 1/4$, which means that we must be in one of the following two cases: 
either $\frak b=[1,-2]$ and $\frak c$ is of Type B (with norm $\ge 17/108$ from the work in Section 4.3), or $\frak b=[1,-3]$ and $\frak c=[1,-2,4]$.   
In both cases $\widetilde{\frak b} +\widetilde{\frak c}$ has length odd and at least $3$, so that its norm is $\ge 1/8$.  
Therefore in the first case we have 
$$ 
N(\frak a) \ge \frac{4}{5} \Big( \frac{1}{12} + \frac{17}{108}\Big) + \frac{1}{5}\frac{1}{8}  > \frac{5}{24},  
$$ 
and in the second case we have 
$$ 
N(\frak a) \ge \frac{4}{5} \Big( \frac{1}{9} + \frac{1}{8} \Big) + \frac{1}{5} \frac{1}{8}  >\frac{5}{24}.  
$$ 
\end{proof} 

\subsection{Length $6$}   

\begin{lemma} \label{lem7.3}  The primitive lists of length $6$ of the form $[a,-2a, b,-2b, c,-2c]$ with norm $< 11/60$ are given as follows: 
$$ 
\text{Norm } \tfrac{17}{108}: \ \ [1,-2,-3,6,9,-18];
$$
$$
\text{Norm } \tfrac{31}{180}: \ [1,-2,-3,-5,6,10], \ [3,5,-6,-10,-15,30];
$$ 
$$ 
\text{Norm } \tfrac{13}{72}:  \  \ [1,-2,4,-8,-12,24], \ [3,-4,-6,8,12,-24],
$$ 
$$ 
\hskip 1 in \ [1,-2,-3,4,6,-8], \ [1,-2,-3,6,-12,24]; 
$$ 
$$ 
\text{Norm } \tfrac{5}{28}: \ \ [1,-2,-3,6,-7,14], \ [1,-2,-7,14,21,-42]; 
$$ 
$$
\text{Norm } \tfrac{59}{324}: \ \ [1,-2,-9,18,27,-54]. 
$$
\end{lemma}
\begin{proof}    Suppose exactly one of $a$, $b$, $c$ is odd (say $a$).  Then, by the involution of Section 4.1, this list has the same norm as 
$[-a,b,-2b,c,-2c]$, and we can use our work from Section 7.1 in tabulating these six term lists with small norm.   This contributes the first two lists with norm $13/72$ given above.  
 
Now suppose $a$ and $b$ are odd, but $c$ is even.  Applying the involution, this has the same norm as  $[-a,-b,c,-2c]$, and from our work in Section 4.4, we 
obtain the second two lists  with norm $13/72$. 

Finally if $a$, $b$, and $c$ are all odd, then the involution gives $[-a,-b,-c]$, with all entries odd.  From our work in Section 4.3 we 
can find all such lists with norm below $11/60$, obtaining all the remaining lists given in the lemma. 
 \end{proof}

\begin{lemma} \label{lem7.4}   If $\frak a$ is a primitive list of length $6$ not of the form $[a,-2a,b,-2b,c,-2c]$ then $N(\frak a) >7/36$ except 
for 
$$ 
\text{Norm } \tfrac 16: \ \ [1,-2, -3,4,6,-12]; 
$$ 
$$
\text{Norm } \tfrac{7}{36}: \ \ [1,-2,-3,6,8,-24], \ [1,-3,-4,8,12,-24].  
$$ 
\end{lemma} 
\begin{proof}  If the elements of $\frak a$ are not divisible by any prime apart from $2$, then by Lemma \ref{lem3.3} we have $N(\frak a)> \frac{7}{36}$.  
Thus we may assume that $\frak a$ is at least $3$--separated.     If $\frak a$ is split into lists of length $1$ and $5$ then 
$$ 
N(\frak a) \ge \tfrac 23 ( \tfrac{1}{12} + \tfrac{1}{6} ) + \tfrac{1}{3} \tfrac{1}{9} > \tfrac{7}{36}. 
$$ 

Now suppose that $\frak a$ is at least $4$--separated, and splits into $\frak b$ with length $2$ and $\frak c$ with length $4$.  
Then $N(\frak a) >7/36$ unless $N(\frak b) + N(\frak c) \le 25/108$.  
Now note that one of $\frak b$ or $\frak c$ must be of Type B.  Using our knowledge of lists with length $2$ and $4$ (see Sections 4.2 and 4.4), this 
forces $\frak b = [1,-3]$ and $\frak c= [1,-2,-3,6]$.  Thus $\frak a$ is of the form $[a,-2a,-3a,6a,b,-3b]$, and a small calculation shows that 
the two lists with norm $7/36$ given in the lemma are the only possibilities.  

Now suppose that $\frak a$ is at least $4$--separated, and splits into two lists of length $3$.   If $N(\frak a) \le 7/36$ we must then have $N(\frak b) + N(\frak c) \le 7/36 \times 4/3= 7/27$, 
and this forces $\frak b = \frak c= [1,-2,4]$ (by our work in Section 4.3).   Therefore $\frak a$ is of the form $[a,-2a,4a,b,-2b,4b]$, and a small 
calculation shows that the only possibility is the list of norm $1/6$ given in the lemma.  

It remains lastly to consider the case when $\frak a$ is exactly $3$--separated, and splits either into lists of length $2$ and $4$, or into two lists of 
length $3$.  Such lists have at at least two elements that are powers of $2$ (up to sign), and at least one element that is $\pm 3$ times a power of $2$, 
and with all elements being divisors of $2^5 \times 3^4$.   Direct computer calculation of the norms of such lists now verifies the lemma.   
\end{proof} 


\section{Classifying integral factorial ratios of length $7$} 


\subsection{Finding the lists}   First we find the twenty one examples of integral factorial ratios of length $7$. 
 Eighteen of these twenty one examples are Type A lists of the form $[a,-2a, b, -2b, c, -2c, d=a+b+c]$.  
We first found, by a computer calculation, all primitive Type A lists that are at most $7$--separated; by a variant of Proposition \ref{prop2.4} 
all elements of such lists  are divisors of $2^6\times 3^2\times 5^2 \times 7^2$.   A straightforward computer program allows us to enumerate all such solutions.  These are: 
%
$$ 
[4, 6, 14, -2, -3, -7, -12] ; \qquad [1, 6, 10, -2, -3, -5, -7]; \qquad [1,6, 20, -2, -3, -10, -12]
$$ 
$$
[4,5,30, -2, -10, -12, -15]; \qquad [3, 4, 18, -2, -6, -8, -9]; \qquad [1, 4, 24, -2, -7, -8, -12]
$$
$$
[2, 7, 20, -1, -4, -10, -14]; \qquad [4, 7, 24, -1, -8, -12, -14]; \qquad [3, 5, 30, -6, -7, -10, -15] 
$$ 
$$
[3,5, 18,-1, -6, -9, -10]; \qquad [2, 3, 12, -1, -4, -6, -6]; \qquad [2, 5, 24, -1, - 8, -10, -12] 
$$ 
$$
[2,3,18,-1,-6,-7,-9]; \qquad [2, 3, 20,-1,-6, -8, -10]; \qquad [5, 9, 30, -1, -10, -15, -18]
$$
$$ 
[4, 5, 30, -6, -8, -10, -15]; \qquad [2, 3, 30, -4, -6, -10, -15]; \qquad [6, 9, 20, -3, -4, -10,-18]. 
$$

Similarly, we consider lists of the form $[a,-2a, b, -2b, c, -3c, d=a+b+2c]$ that are at most $5$ separated.  
Checking these lists on a computer, we discover three more solutions: 
$$ 
[3, 5, 30, -1, -10, -12, -15]; \qquad [1, 6,15, -2, -3, -5, -12]; \qquad [3, 4, 24, -2, -8, -9, -12]. 
$$ 

It remains now to show that there are no further solutions, and we split this into two cases: when $\frak a$ is of Type A (and we 
may assume at least $8$--separated), and when $\frak a$ is of Type B.     

\subsection{The case $\frak a$ of Type A}  Since our computer calculation covers all Type A examples that are at most $7$--separated, we 
may assume that $\frak a$ is at least $8$--separated.  If $\frak b$ and $\frak c$ have their usual meanings, note that since $\frak a$ is of Type A 
one must have both $\frak b$ and $\frak c$ being of Type A (else property (iii) in Definition \ref{def2.1} cannot be met).   As before, our argument 
divides into cases depending on the lengths of $\frak b$ and $\frak c$.

Suppose that $\frak b=[1]$ and $\frak c$ (which is of Type A) has length $6$.   
If $N(\frak c) <\frac {11}{60}$ then we  can use our work in Section 7.2 to obtain all such Type A lists, and directly check that 
there are no new solutions.  If $N(\frak c)\ge11/60$ then (noting that $\ell(\widetilde{\frak b}+\widetilde{\frak c}) = 5$ or $7$)
$$ 
N(\frak a) \ge \frac{7}{8} \Big( \frac{1}{12}+ \frac{11}{60} \Big) + \frac{1}{8} \frac{1}{6} > \frac 14,
$$ 
and again we are done.  

Next suppose $\ell(\frak b)=2$ (so that the Type A list $\frak b$ must be $[1,-2]$) and $\ell(\frak c) =5$.  
   In section 7.1 we obtained all the Type A lists of length $5$ with norms at most $31/168$, and checking these 
cases we found no new examples.  If the norm of the length $5$ list is $>31/168$ then 
$$ 
N(\frak a) > \frac 78 \Big( \frac{1}{12} + \frac{31}{168}\Big) + \frac 18 \frac 18 = \frac 14,
$$ 
and so we are done with this case. 

We are left with the last case $\ell(\frak b) = 3$ and $\ell(\frak c)=4$.   Since $N(\frak b) \ge 1/8$, we must have 
  $N(\frak c) \le 25/168$, else 
$$ 
N(\frak a) > \frac 78 \Big( \frac 18 +\frac {25}{168} \Big) + \frac 18 \frac 1{12} =\frac 14. 
$$ 
Similarly, since $N(\frak c) \ge 1/9$ we find that $N(\frak b) \le 41/252$.  Our work in Sections 4.3 and 4.4 gives us all 
the Type A lists of lengths $3$ and $4$ with norms in these ranges, and checking these we verify that no new solutions are 
obtained.  

This completes our treatment of Type A lists.

\subsection{The case $\frak a$ of Type B}  If $\frak a$ is at most $4$ separated then (by a variant of Proposition \ref{prop2.4}, using the additional 
fact $s(\frak a)=0$) the elements of $\frak a$ must be divisors of $2^{10} \times 3^5$.  By a computer calculation 
 we checked   that there are no new integral factorial ratios arising from such Type B lists.  
   So we may assume that $\frak a$ is of Type B and is $k$--separated 
with $k\ge 5$.  Let $\frak b$ and $\frak c$ have their usual meanings, and note that at least one of $\frak b$ or $\frak c$ must be of 
Type B.  

If $\frak b=[1]$ and $\ell(\frak c)=6$, then 
$$ 
\tfrac 14 = N(\frak a) \ge \tfrac 45 \big( \tfrac{1}{12} + N(\frak c)\big) + \tfrac{1}{5} \tfrac 16,
$$ 
which forces $N(\frak c) \le 9/48$. There is only one six term list of Type  B with norm below $9/48$, 
namely $[1,-2,-3,6,8,-24]$ and this does not give a solution. 

If $\ell(\frak b) = 2$ and $\ell(\frak c)=5$ then 
$$ 
\tfrac 14 = N(\frak a) \ge \tfrac 45 \big( \tfrac{1}{12} + N(\frak c) \big) + \tfrac 15 \tfrac 18, 
$$ 
which forces $N(\frak c) \le 19/96$, and our work in Section 7.1 shows that there are no Type B lists 
of length $5$ with such small norm.  So $\frak c$ is of Type A, which forces $\frak b$ to be of Type B whence $N(\frak b)\ge 1/9$.  Repeating our 
argument we obtain that $N(\frak c) \le 49/288$, leaving us with only the four lists of length $5$ and norm $1/6$.  
Further given one of those lists, a similar calculation gives $N(\frak b) \le 11/96$, so that $\frak b$ is forced to be $[1,-3]$.  
A quick check of these four cases produced no new examples.

Lastly consider the case $\ell(\frak b)=3$ and $\ell(\frak c)=4$.  If $\frak c$ is of Type B, then using $N(\frak b) \ge 1/8$ we 
can check that $N(\frak c) \le 1/6$, forcing $\frak c$ to be one of the $9$ lists of norm $1/6$, and in turn forcing $\frak b$ to be $[1,-2,4]$.  
These cases are easily checked.  On the other hand, if  $\frak b$ is of Type B, then using $N(\frak c) \ge 1/9$ we are 
forced to have $N(\frak b) \le 13/72$.   In turn, since $\frak b$ is of Type B, we must have $N(\frak b) \ge 17/108$ which forces $N(\frak c) \le 29/216$.  So once again 
we know all the possible choices for $\frak b$ and $\frak c$, and a simple check reveals that these lead to no new solutions.

\section{Understanding lists of length $7$ and $8$} 


\noindent We pave the way for classifying integral factorial ratios of length $9$ in the next section, by determining $G(7)$ and $G(8)$ here.

\begin{lemma} \label{lem9.1}  We have $G(7) = 5/24$ and $G(7;1)  =17/72$.  
\end{lemma} 
\begin{proof}  From our knowledge of $G(n)$ for $n\le 6$ and Proposition \ref{prop3.1} we 
may easily evaluate $G(7;1)$.   Now we turn to the evaluation of $G(7)$.  
 
 First note that if the list $\frak a$ is at most $2$--separated then by Lemma \ref{lem3.3}, $N(\frak a) \ge 89/384 > 5/24$.  
 So we may now suppose that $\frak a$ is $k$--separated with $k\ge 3$, and maintain our usual notation.  

If $\ell(\frak b) =1$ and $\ell(\frak c)=6$ then (using that $\widetilde{\frak b} + \widetilde{\frak c}$ has length $5$ or $7$, and so $N(\widetilde{\frak b}+\widetilde{\frak c}) \ge 1/6$)  
$$ 
N(\frak a) \ge \tfrac 23 \big( \tfrac{1}{12} + \tfrac{17}{108}\big) + \tfrac{1}{3}\tfrac16 = \tfrac{35}{162} > \tfrac{5}{24}. 
$$ 

If $\ell(\frak b)=2$ and $\ell(\frak c)=5$ then 
$$ 
N(\frak a) \ge \tfrac 23 \big( \tfrac{1}{12} +\tfrac 16\big) + \tfrac 13\tfrac 18 = \tfrac{5}{24},  
$$ 
and this is attained in the example $[1,-2,-3,6,9,-18,36]$.

The last case $\ell(\frak b)=3$ and $\ell(\frak c)=4$ is   more involved.  Suppose first that $\frak a$ is 
of Type A, so that both $\frak b$ and $\frak c$ are of Type A.  First we checked the cases when $\frak a$ is at most $4$--separated by direct computation 
(the elements of such lists must be divisors of $2^9 \times 3^3$).   Suppose therefore that $\frak a$ is of Type A and 
at least $5$ separated.  If either $\frak b\neq [1,-2,4]$ or $\frak c \neq [1, -2, -3, 6]$ then we have 
$N(\frak b) + N(\frak c) \ge \min (\frac 19 +\frac{5}{36}, \frac 18 + \frac{2}{15}) = \frac 14$ and so 
$$
N(\frak a) \ge \tfrac 45 \tfrac 14 + \tfrac{1}{5} \tfrac{1}{12} > \tfrac 5{24}. 
$$ 
If $\frak b = [1,-2,4]$ and $\frak c = [1,-2,-3,6]$ then $\ell(\widetilde{\frak b} + \widetilde{\frak c}) =3$, $5$, or $7$ and so $\widetilde{\frak b} +\widetilde{\frak c}$ has 
norm at least $1/8$.   Therefore
$$ 
N(\frak a) \ge \tfrac 45 \big( \tfrac 18+\tfrac 19\big) + \tfrac 15 \tfrac 18 > \tfrac{5}{24}. 
$$ 

Now suppose that $\frak a$ is of Type B, so that at least one of $\frak b$ or $\frak c$ must be of Type B.   If both $\frak b$ and $\frak c$ are of Type B then 
$$ 
N(\frak a) \ge \tfrac 23 \big( \tfrac{17}{108} + \tfrac 16 \big) +\tfrac 1{3} \tfrac{1}{12} > \tfrac 5{24} .
$$
If exactly one of $\frak b$ and $\frak c$ is of Type A and the other is of Type B then 
$$ 
N(\frak b) +  N( \frak c) \ge \min \big( \tfrac 18 + \tfrac 16, \tfrac {17}{108} +\tfrac{1}{9} \big) = \tfrac{29}{108}.
$$ 
Further, here we must have $\ell(\widetilde{\frak b} + \widetilde{\frak c})$ being odd and at least $3$ so that 
$N(\widetilde{\frak b} + \widetilde{\frak c} ) \ge 1/8$.  Therefore  
$$ 
N( \frak a ) \ge \tfrac 23 \tfrac{29}{108} + \tfrac 13 \tfrac 18 > \tfrac {5}{24}, 
$$ 
completing our proof.  
 \end{proof}

\begin{lemma} \label{lem9.2}  We have $G(8) = 8/45$ and $G(8;1)=2/9$.  
\end{lemma} 
\begin{proof}  Determining $G(8;1)$ is easy by Proposition \ref{prop3.1}, and we focus on evaluating $G(8)$.   Note that the list 
$[1,-2,-3,6,-5,10,15,-30]$ has norm $8/45$, and we need only show that no smaller norm is possible.   Let 
$\frak a$ be a primitive list of length $8$ and let $p$ denote the largest prime dividing some element of $\frak a$. 
Clearly $\frak a$ is $p$--separated, and divide $\frak a$ into the sublists consisting of the multiples of $p$ and the non-multiples of $p$. 

If $p=2$ then by Lemma \ref{lem3.3} we have $N(\frak a) \ge 199/768 > 1/4$ and there is nothing to prove.  
If $p\ge 5$ then $N(\frak a) \ge \frac 45 \frac 29 =\frac{8}{45}$ because $G(i) + G(8-i) \ge \frac 29$ for all $1\le i<8$. 
Thus it remains only to consider the case $p=3$.  

If $\ell(\frak b) =1$, $2$ or $3$ (and so $\ell( \frak c) =7$, $6$, or $5$)  then 
 $$ 
N(\frak b) + N(\frak c) \ge \min \big( \tfrac{1}{12} + \tfrac{5}{24}, \tfrac{1}{12}+ \tfrac{17}{108}, \tfrac{1}{8} +\tfrac 16\big) =\tfrac{13}{54}, 
$$
so that 
$$
N(\frak a) \ge \tfrac{2}{3} \tfrac{13}{54} + \tfrac{1}{3} \tfrac{1}{12} > \tfrac{8}{45}. 
$$ 

We are left with the case $p=3$ and   $\frak b$ and $\frak c$ both have length $4$.  One of these two lists of 
length $4$ must contain only powers of $2$ and hence has norm $\ge 7/48$ by Lemma \ref{lem3.3}.  If $\frak b =\frak c$ then we find $N(\frak a) \ge \frac 23 (N(\frak b) +N(\frak c)) 
\ge \frac{7}{36}$, which suffices.   If $\frak b \neq \frak c$ then $\widetilde{\frak b} + \widetilde{\frak c}$ is non-empty and we have 
$$ 
N(\frak a) \ge \tfrac 23 \big( \tfrac 19 + \tfrac{7}{48}\big) + \tfrac{1}{3} \tfrac{1}{12} > \tfrac{8}{45}. 
$$
\end{proof}

\section{Classifying integral factorial ratios of length $9$}  


\subsection{Lists that are at least $5$ separated}  We first determine the integral factorial ratios arising from primitive lists $\frak a$ of length $9$ that 
are $k\ge 5$ separated.  Let $\frak b$ and $\frak c$ have their usual meanings.   

If $\ell(\frak b)=1$ and $\ell(\frak c)=8$ then $N(\frak b) \ge 1/12$ and by Lemma \ref{lem9.2} $N(\frak c) \ge 8/45$.  
Further $N(\widetilde{\frak b} +\widetilde{\frak c}) \ge 5/24$ and so we conclude that 
$$
N(\frak a) \ge \tfrac 45\big( \tfrac{1}{12} + \tfrac{8}{45} \big) + \tfrac{1}{5} \tfrac{5}{24} > \tfrac 14. 
$$ 
Similarly if $\ell(\frak b)=2$ and $\ell(\frak c)=7$  then 
$$ 
N(\frak a) \ge \tfrac{4}{5} \big( \tfrac{1}{12} + \tfrac{5}{24}\big) + \tfrac{1}{5} \tfrac{1}{6} > \tfrac{1}{4}. 
$$ 
If $\ell(\frak b)=3$ and $\ell(\frak c)=6$ then 
$$ 
N(\frak a) \ge \tfrac 45 \big( \tfrac{1}{8} + \tfrac{17}{108} \big) + \tfrac{1}{5} \tfrac{1}{8} > \tfrac 14. 
$$ 

Lastly suppose that $\ell(\frak b) = 4$ and $\ell(\frak c)=5$.  Here we must have $N(\frak b) + N(\frak c) \le 7/24$, else 
$N(\frak a)$ would be $>\frac 14$.   Since $N(\frak c) \ge \frac 16$, this inequality implies that $N(\frak b) \le 1/8$ which 
forces $\frak b = [1,-2,-3,6]$ and $N(\frak b) = 1/9$.  In turn, since $\frak b$ is known, we infer that $N(\frak c) \le 13/72$.  
Now from our work in Section 7.1 we know all the primitive lists of length $5$ with norm at most $13/72$.  Checking these 
lists we discover exactly one factorial ratio arising in this manner: namely, 
$$ 
[2,3,5,30,-1,-6,-8,-10,-15]. 
$$ 

\subsection{Lists that are at most $4$ separated}   First we checked all   lists of length $9$ with sum $0$, of Type A that are at most $4$ separated. 
This produced one more integral factorial ratio:  namely, 
$[4,6,9,24,-2,-3,-8,-12,-18]$.   Now we show that there are no further solutions.  We assume below that $\frak a$ is of Type B and at most $4$ separated.

 If all the elements of $\frak a$ are powers of $2$, then by Lemma \ref{lem3.3} we know that $N(\frak a) \ge 147/512 >\frac 14$.   Therefore there 
 must be some elements of $\frak a$ that are multiples of $3$, so that $\frak a$ is $3$--separated,  and we split $\frak a$ into the multiples of $3$ and the non-multiples of $3$, and 
 keep our usual meanings for $\frak b$ and $\frak c$.  
 
 Suppose $\ell(\frak b)=1$ and $\ell(\frak c)=8$.  Then $\frak b=[1]$ so that $B{\frak b}$ must be the multiple of $3$, and so $C\frak c$  (and hence $\frak c$) 
 consists only of powers of $2$.  
 From Lemma \ref{lem3.3} it follows that $N(\frak c) \ge 199/768$ whence $N(\frak a) \ge \frac 23 ( \frac{1}{12} + \frac{199}{768}) + \frac{1}{3} \frac{5}{24} > \frac 14$.  
 
 Next suppose $\ell(\frak b)=2$ and $\ell(\frak c)=7$.    If $\frak b \neq [1,-2]$ then $N(\frak b) \ge \frac{1}{9}$, and 
 $$ 
 N(\frak a) \ge \tfrac 23 \big( \tfrac 19 + \tfrac{5}{24} \big) + \tfrac 13 \tfrac 16 > \tfrac 14. 
 $$ 
 If $\frak b =[1,-2]$ then $B\frak b$ must be the multiples of $3$, and therefore $\frak c$ consists of just powers of $2$ and 
 $N(\frak c) \ge 89/384$ by Lemma \ref{lem3.3}.  Therefore 
 $$ 
 N(\frak a) \ge \tfrac 23 \big( \tfrac{1}{12} + \tfrac{89}{384} \big) + \tfrac 13 \tfrac 16 >\tfrac 14.
 $$

 Suppose $\ell(\frak b)=3$ and $\ell(\frak c)=6$.  Since $\frak a$ is of Type B, either $\frak b$ or $\frak c$ must be of Type B.  If 
 $\frak b$ is of Type B, then $N(\frak b) \ge \frac{17}{108}$, and we find 
$$ 
N(\frak a) \ge \tfrac 23 \big( \tfrac{17}{108} + \tfrac{17}{108} \big) + \tfrac 13 \tfrac 18 > \tfrac 14. 
$$ 
So we may assume that $\frak b$ is of Type A, and therefore $\frak c$ is of Type B.   Since $N(\frak b) \ge \frac 18$, we find that 
$N(\frak c) \le \frac{3}{16}$ (else one would have $N(\frak a) > \frac 14$), and by Lemma \ref{lem7.4} this forces $\frak c= [1,-2,-3,4,6,-12]$ with norm $\frac 16$.  
In turn, knowing $\frak c$ we deduce that $\frak b$ must have norm at most $\frac{7}{48}$.   From our work in Section 4.3, there are only $7$ possibilities for 
$\frak b$, and checking these possibilities we find no new factorial ratios.  

Finally it remains to consider the case where $\ell(\frak b)=4$ and $\ell(\frak c)=5$.  If both $\frak b$ and $\frak c$ are of Type B then 
$N(\frak b) + N(\frak c) \ge \frac 16 +\frac 5{24} = \frac 3{8}$ and so $N(\frak a) \ge \frac 23 \frac 38 +\frac{1}{36} > \frac 14$.
Suppose now that exactly one of $\frak b$ or $\frak c$ is of Type A and the other is of Type B.  In this case $\widetilde{\frak b} +\widetilde{\frak c}$ 
has odd length at least $3$ and so has norm $\ge \frac 18$.  Therefore 
$$ 
N(\frak a) \ge \tfrac 23 \min \big( \tfrac 19+ \tfrac 5{24}, \tfrac 16 +\tfrac 16 \big) + \tfrac 13 \tfrac 18 = \tfrac 23 \tfrac {23}{72} + \tfrac 1{24} >\tfrac 14. 
$$ 

 
This completes our classification of the integral factorial ratios of length $9$.

  \section{Proof of Theorem \ref{thm1.3}}

\subsection{Determining $G(n;d)$} We may clearly assume that $1\le d <n$.  First suppose that  
  $\frak a_1$, $\ldots$, $\frak a_{d+1}$ are $d+1$ lists whose lengths add up to $n$.  Then for a generic choice of 
  integers $x_1$, $\ldots$, $x_{d+1}$ we may form the list $\frak a = x_1 \frak a_1 + x_2 \frak a_2 + \ldots + x_{d+1} \frak a_{d+1}$, 
  which will typically be a list of length $n$.  Further if we choose $x_1$, $\ldots$, $x_{d+1}$ to be large coprime integers, then 
  the set of such $\frak a$ will escape any finite collection of subspaces of ${\Bbb R}^n$ of dimension at most $d$, and moreover one has 
  $$ 
  N(\frak a) \ge N(\frak a_1 )  + \ldots +N(\frak a_{d+1}) - \epsilon 
  $$ 
  for any $\epsilon >0$.   This argument shows that 
  $$ 
  G(n;d) \le \min_{\ell_1 + \ldots + \ell_{d+1} = n} \big( G(\ell_1)  +\ldots + G(\ell_{d+1}) \big). 
  $$ 
  
  Now we establish the reverse inequality, which would prove the first part of Theorem \ref{thm1.3}.   Let $\epsilon >0$ be given, and suppose $\frak a$ is a 
  list of length $n$, which we may assume (by omitting finitely many primitive lists)  is $k$ separated with $k\ge 1/\epsilon$ being large.   Thus there are 
  primitive smaller lists $\frak b$ and $\frak c$  with $N(\frak a) \ge (1-\epsilon) \big( N(\frak b) + N(\frak c)\big)$.  Let $b$ denote the length of $\frak b$, 
  and suppose that $r$ is the smallest integer with $N(\frak b) \le G(b;r) -\epsilon$.  Then $N(\frak b) \ge G(b;r-1) -\epsilon$, and moreover $\frak b$ must 
  lie in one of finitely many vector spaces of dimension at most $r$; say these vector spaces are $V_1$, $\ldots$, $V_R$, and these vector spaces depend only on $b$, $r$ and $\epsilon$.  Now suppose that $N(\frak c) \le G(n-b;d-r) -\epsilon$.   Then $\frak c$ must lie in one of finitely many vector spaces $W_1$, $\ldots$, $W_S$ (depending only on $n-b$, $d-r$ and $\epsilon$) of dimension at most $d-r$.  But then $\frak a = B \frak b + C \frak c$ will lie in one of finitely many vector spaces of dimension $\le d$; namely 
  in one of the vector spaces arising from the direct sum of $V_j$ and $W_\ell$.  In other words, if we know that $\frak a$ does not lie in one of these finitely many 
  subspaces of dimension $\le d$, then we must have 
  $$ 
  N(\frak a) \ge (1-\epsilon) \big( G(b;r-1)-\epsilon + G(n-b;d-r) -\epsilon \big) = G(b;r-1) + G(n-b;d-r) - O(\epsilon). 
  $$ 
  Removing the vector spaces of dimension $\le d$ that arise as above for all choices of $b$ and $r$, and we conclude that 
  $$ 
  G(n;d) \ge \min_{1 \le b <n} \min_{1\le r \le d} \big( G(b;r-1) + G(n-b;d-r)\big) - O(\epsilon). 
  $$ 
  Letting $\epsilon \to 0$, it follows that 
  $$ 
    G(n;d) \ge \min_{1 \le b <n} \min_{1\le r \le d} \big( G(b;r-1) + G(n-b;d-r)\big).
  $$ 
  By induction we conclude that 
   $$ 
  G(n;d) \ge \min_{\ell_1 + \ldots + \ell_{d+1} = n} \big( G(\ell_1)  +\ldots + G(\ell_{d+1}) \big). 
  $$

  To complete our discussion on $G(n;d)$, since $G(\ell) \ge 1/12$ for all $\ell \ge 1$, clearly $G(n;d) \ge (d+1)/12$, 
 for all $n\ge d+1$.  When $d+1 \le n\le 2d+2$, one can write $n$ using $d+1$ ones and twos, so that the equality 
 $G(n;d) = (d+1)/12$ holds here.  If $n \ge 2d+3$, then one of the $\ell_j$'s that sum to $n$ must be at 
 least $3$, and using Corollary \ref{cor3.5} we conclude that $G(n;d) \ge d/12 +1/9$.  
  
  \subsection{Bounding ${\widetilde G}(n;d)$}    When $d=0$ the stated bound holds trivially, and henceforth 
  assume that $d\ge 1$.  The proof of the lower bound is a variant of the argument given above for $G(n;d)$.  
  By omitting finitely many lists, we may assume that $\frak a$ (of length $n$ and $s(\frak a)=0$) 
  is $k$ separated with $k$ sufficiently large.  Then there are smaller lists $\frak b$ and $\frak c$ with lengths 
  $b$ and $c$ (with $n=b+c$) such that $N(\frak a) \ge (1-\epsilon) (N(\frak b) + N(\frak c))$, and $\frak a = B\frak b +C \frak c$.   
  Now there are two cases: either $s(\frak b)$ and $s(\frak c)$ are both non-zero, or $s(\frak b) = s(\frak c) =0$.

  Consider first the case when both $s(\frak b)$ and $s(\frak c)$ are non-zero.  Let $r$ be the smallest integer with 
  $N(\frak b) \le G(b;r) -\epsilon$ so that $N(\frak b) \ge G(b;r-1) -\epsilon$, and moreover $\frak b$ must lie in one 
  of finitely many vector spaces of dimension at most $r$.  Let these subspaces be $V_1$, $\ldots$, $V_R$ 
  and we can assume that each of these subspaces is not contained in the hyperplane $s(\frak x) =0$.  Now suppose $N(\frak c) 
  \le G(n-b;d-r+1)-\epsilon$.  Then $\frak c$ must lie in one of finitely many subspaces $W_1$, $\ldots$, $W_S$ of dimension 
  at most $d-r+1$.  But then $\frak a = B \frak b+ C\frak c$ lies in a vector space arising as 
  the direct sum of $V_j$ and $W_\ell$  intersected with the hyperplane $s(\frak x) =0$, and these subspaces 
  have dimension $\le r + d-r+1 -1 = d$.  This is exactly as in our previous argument, with the added benefit 
  of intersecting with the hyperplane $s(\frak x) =0$ which allows for dimension $\le d-r+1$ in the treatment of $\frak c$ instead 
  of our earlier $\le d-r$.  Summarizing, in the case when $s(\frak b)$ and $s(\frak c)$ are non-zero, after removing finitely many 
  vector spaces of dimension at most $d$, we can conclude that 
  $$ 
  N(\frak a) \ge \min_{b, r}  \big( G(b;r-1) + G(n-b;d-r+1) \big) - O(\epsilon) = G(n;d+1) - O(\epsilon), 
  $$ 
  by our work on $G(n;d)$ above.

  Now we turn to the second case where $s(\frak b)  = s(\frak c)=0$.  Let now $r$ denote the smallest integer 
  such that $N(\frak b) \le \widetilde{G}(b;r) -\epsilon$, so that $N(\frak b) \ge \widetilde{G}(b;r-1) -\epsilon$ 
  and $\frak b$ lies in one of finitely many vector spaces of dimension at most $r$.  Now if 
  $N(\frak c) \le \widetilde{G}(n-b;d-r) -\epsilon$, then $\frak c$ would have to lie in one of finitely many vector 
  spaces of dimension at most $d-r$, and we would be able to conclude that $\frak a = B\frak b + C \frak c$ 
  must lie in one of finitely many vector spaces of dimension at most $d$.  Thus, in this case we can conclude that 
  after removing finitely many vector spaces of dimension at most $d$, 
  $$ 
  N(\frak a) \ge \min_{b,r} \big( \widetilde{G}(b;r-1) + \widetilde{G}(n-b;d-r) \big). 
  $$ 
  
  Putting both arguments together and letting $\epsilon \to 0$ it follows that 
  $$
  \widetilde{G}(n;d) \ge \min_{b,r}  \big( G(n;d+1) , \widetilde{G}(b;r-1) + \widetilde{G}(n-b;d-r)\big). 
  $$ 
By induction one deduces that 
\begin{equation} 
\label{11.1} 
    \widetilde{G}(n;d) \ge \min \big( G(n;d+1), \min_{\ell_1 +\ldots +\ell_{d+1} =n} \widetilde{G}(\ell_1) + \ldots + \widetilde{G}(\ell_{d+1}  )\big).
\end{equation}

  Now the second term in the right side of \eqref{11.1} is only relevant when all the $\ell_j$ are at least $3$ (because there are no 
  non-degenerate lists of length $1$ or $2$ with sum $0$) .  Thus when $n\le 3d+3$, the bound is simply $\widetilde{G}(n;d) \ge G(n;d+1)$.  In 
  any event, the second term in the right side of \eqref{11.1} exceeds $(d+1)/9$, and so the bounds stated for $\widetilde{G}(n;d)$ for small 
  values of $n$ follow.

  \section{Proof of Theorem \ref{thm1.4}}  
  
  \noindent From \eqref{3.7} we already know that $G(n)$ is as large as the asymptotic stated in the theorem.  It 
  remains therefore to establish the upper bound.  The following lemma gives a natural construction of examples 
  built out of the Liouville function.  
   
 \begin{lemma} \label{lem12.1}   Given a positive integer $N$ denote by $\frak L(N)$ the 
 list with elements $\lambda(d) d$ where $d$ runs over the divisors of $N$ and $\lambda(d)$ is the 
 Liouville function $(-1)^{\Omega(d)}$.  Then $\frak L(N)$ is a list with $d(N)$ elements 
 and 
 $$ 
 N(\frak L(N))  =  \frac{d(N)}{12} f(N), 
 $$ 
 where $f(N)$ is a multiplicative function defined on prime powers $p^k$ by 
 $$ 
 f(p^k)  = 1  + 2 \sum_{j=1}^{k} \frac{(k+1-j)}{k+1} \frac{(-1)^j}{p^j}. 
 $$ 
 \end{lemma}  
 \begin{proof}  From \eqref{2.1} we have 
$$ 
N(\frak L(N)) = \frac{1}{12} \sum_{a, b|N} \lambda(a) \lambda(b) \frac{(a,b)^2}{ab} = \frac 1{12} \prod_{p^k \Vert N} \Big( \sum_{r, s= 0}^{k} (-1)^{r+s} 
\frac{p^{2\text{min}(r,s)}}{p^{r+s}} \Big), 
$$  
 where the second relation follows from multiplicativity, denoting by $p^k$ the exact power of $p$ dividing $N$, and by $p^r$ and $p^s$ the 
 corresponding powers of $p$ dividing $a$ and $b$.   Now 
 $$ 
 \sum_{r,s= 0}^k  (-1)^{r+s} 
\frac{p^{2\text{min}(r,s)}}{p^{r+s}}  = (k+1) + 2\sum_{j=1}^{k} \sum_{0\le r \le k-j} \frac{(-1)^j}{p^j} = (k+1) f(p^k), 
 $$
 and the lemma follows.  
 \end{proof}

 Let $k$ be large, and define $N_k$ as follows:
 $$ 
 N_k = \prod_{p\le k} p^r \qquad \text{  with  } r = \lfloor 2^{k/\pi(k)} \rfloor. 
 $$ 
A little calculation using the prime number theorem gives  $2^k \le d(N_k)  = (r+1)^{\pi(k)} < 2^{k+1}$.  
If we denote $d(N_k)$ by $n(k)$, then Lemma \ref{lem12.1} establishes that 
$$ 
G(n(k)) \le \frac{n(k)}{12} \prod_{p\le k} f(p^r) =\frac{n(k)}{12} \prod_{p\le k} \Big( 1- \frac{2}{p+1}  + O\Big( \frac{1}{r p} \Big) \Big) 
\sim \frac{n(k)}{12} \frac{\pi^2 e^{-2\gamma}}{6 (\log \log n(k))^2}. 
$$ 
Thus the asymptotic upper bound holds for a sequence $n(k)$ with $n(k) \in [2^k, 2^{k+1})$ for each large $k$.   For a general 
large $n$, use a greedy procedure to express $n$ as a sum of elements from the sequence $n(k)$ up to some bounded error.  The 
result now follows from the sub-additivity of the function $G$ (see Proposition \ref{prop3.1}).

 \section{Proof of Theorem \ref{thm1.2}}

 \noindent Theorem \ref{thm1.2} can be deduced easily from our other results.   Let $\frak a$ be a list of length $K+L$ 
 corresponding to a factorial ratio with $K$ terms in the numerator and $L=K+D$ terms in the denominator.  Then we must 
 have, using Theorem \ref{thm1.4},   
 $$ 
\frac{D^2}{4} \ge N(\frak a) \ge G(K+L) \sim \frac{\pi^2}{72 e^{2\gamma}} \frac{(K+L)^2}{(\log \log (K+L))^2}. 
$$ 
From this, part 1 of Theorem \ref{thm1.2} follows.  Note that one can also use the explicit bound in \eqref{3.6} to calculate 
numerical bounds for $K+L$ for any specified value of $D$.

 Theorem \ref{1.3} shows that ${\widetilde G}(K+L;3D^2-1) \ge (3D^2+1)/12 = D^2/4 +1/12$.   Thus if $\frak a$ of length $K+L$ 
 does not lie in finitely many vector spaces of dimension at most $3D^2-1$ then $N(\frak a) > D^2/4$, which proves part 2 of the theorem.  
 
Corollary \ref{cor3.5} shows that $G(n) > 1$ for $n\ge 82$, so that there are no factorial ratios with $D=2$ and $K+L\ge 82$.  Further the 
numerical table produced for Corollary \ref{cor3.5} also shows that $G(n;1) >1$ for $n\ge 76$, which proves the last assertion of Theorem \ref{thm1.2}.

  \ \

 \end{document}